\documentclass[11pt,letterpaper,reqno]{amsart}
\usepackage[margin=1in]{geometry}

\usepackage[usenames,dvipsnames]{xcolor}
\usepackage{fancyhdr}
\usepackage{amsmath,amsfonts,amsbsy,amsgen,amscd,amssymb,amsthm}
\usepackage{url}

\usepackage{mathtools}

\mathtoolsset{showonlyrefs}

\usepackage[font=small,margin=0.25in,labelfont={bf},labelsep={colon}]{caption}

\usepackage{subcaption}

\usepackage{tikz}
\usepackage{microtype}
\usepackage{enumitem}

\definecolor{dark-gray}{gray}{0.3}
\definecolor{dkgray}{rgb}{.4,.4,.4}
\definecolor{dkblue}{rgb}{0,0,.5}
\definecolor{medblue}{rgb}{0,0,.75}
\definecolor{rust}{rgb}{0.5,0.1,0.1}

\usepackage[colorlinks=true]{hyperref}

\hypersetup{urlcolor=rust}
\hypersetup{citecolor=dkblue}
\hypersetup{linkcolor=dkblue}

\usepackage{setspace}

\usepackage{graphicx}
\usepackage{booktabs,longtable,tabu} 
\usepackage{multirow} 

\usepackage{float}

\usepackage[full]{textcomp}

\usepackage[scaled=.98,sups,osf]{XCharter}
\usepackage[scaled=1.04,varqu,varl]{inconsolata}
\usepackage[type1]{cabin}
\usepackage[charter,vvarbb,scaled=1.07]{newtxmath}
\usepackage[cal=boondoxo]{mathalfa}
\usepackage[T1]{fontenc}

\usepackage{bm} 

\graphicspath{{figures/}}

\newtheorem{theorem}{Theorem}[section]
\newtheorem{lemma}[theorem]{Lemma}

\newtheorem{proposition}[theorem]{Proposition}

\theoremstyle{definition}

\newtheorem{remark}[theorem]{Remark}

\numberwithin{equation}{section} 
\numberwithin{figure}{section}
\numberwithin{table}{section}

\floatstyle{plaintop}
\newfloat{recipe}{thp}{lor}
\floatname{recipe}{Recipe}
\numberwithin{recipe}{section}

\renewcommand{\phi}{\varphi}

\providecommand{\mathbbm}{\mathbb} 

\newcommand{\R}{\mathbbm{R}}

\newcommand{\diff}{\mathrm{d}}
\newcommand{\idiff}{\, \diff}

\newcommand{\triplenorm}[1]{{\left\vert\kern-0.25ex\left\vert\kern-0.25ex\left\vert #1
    \right\vert\kern-0.25ex\right\vert\kern-0.25ex\right\vert}}

\newcommand{\ind}{\mathbbm{1}} 

\newcommand{\calM}{\mathcal{M}}

\DeclareMathOperator{\curl}{curl}

\DeclareMathOperator{\supp}{supp}

\DeclareMathOperator*{\esssup}{ess\,sup}

\def\XXint#1#2#3{{%
\setbox0=\hbox{$#1{#2#3}{\int}$}
\vcenter{\hbox{$#2#3$}}\kern-.5\wd0}}

\renewcommand{\leq}{\leqslant}
\renewcommand{\geq}{\geqslant}
\renewcommand{\subset}{\subseteq}

\newcommand{\RR}{\mathbb{R}}

\pdftrailerid{}
\newcommand{\Mnorm}[2]{\lVert #1\rVert_{\mathcal M^{#2}}}
\newcommand{\weaknorm}[2]{\lVert #1\rVert_{L^{#2,\infty}}}
\newcommand{\qw}{q_\omega}
\newcommand{\qu}{q_u}
\newcommand{\qwp}{q_\omega'}
\newcommand{\Idmap}{\operatorname{Id}}
\newcommand{\Riesz}{\mathcal R}
\newcommand{\loc}{\mathrm{loc}}
\title{Self-similar solutions
       of the $2$D incompressible Euler equations\\ via A Lagrangian construction}

\author{Thomas Y. Hou}
\address{Applied and Computational Mathematics,
California Institute of Technology, Pasadena, CA, USA}
\email{hou@cms.caltech.edu}

\author{Xiang Qin}
\address{Applied and Computational Mathematics,
California Institute of Technology, Pasadena, CA, USA}
\email{xqin2@caltech.edu}

\author{Peicong Song}
\address{Applied and Computational Mathematics,
California Institute of Technology, Pasadena, CA, USA}
\email{psong2@caltech.edu}

\author{Wenhan Zhang}
\address{Applied and Computational Mathematics,
California Institute of Technology, Pasadena, CA, USA}
\email{wzhang9@caltech.edu}

\date{}
\begin{document}
\maketitle
\begin{abstract}
We give a Lagrangian construction of forward self-similar solutions
of the $2$D incompressible Euler equation. For every $1<\alpha<2$
and every real $h\in L^{2/\alpha}(\mathbb S^1)$, the construction
produces a flow with initial vorticity $|x|^{-\alpha}h(\arg x)$,
without smallness or sign assumptions. The argument is based on
a fixed-point formulation of backward particle transport from
time $t=1$ to $t=0$. Mollifying the Biot--Savart velocity in
self-similar variables preserves exact vorticity transport,
while symmetrization yields the unforced Euler equation in
velocity form even when the ordinary vorticity flux is not
locally integrable. The limiting solution admits an
almost-everywhere invertible, measure-preserving regular
Lagrangian flow from time zero, with strong local continuity
and preservation of signed value distributions and absolutely
integrable Casimirs.
\end{abstract}

\section{Introduction and main result}\label{sec:introduction}
We consider the two-dimensional incompressible Euler equations
on the whole plane in velocity-pressure form,
\begin{equation}\label{eq:Euler}
 \partial_t u+\operatorname{div}(u\otimes u)+\nabla P=0,
 \qquad
 \operatorname{div}u=0.
 \tag{2D Euler-vel}
\end{equation}
For smooth solutions, the scalar vorticity
$\omega=\curl u$ satisfies the transport equation
\begin{equation}\label{eq:vorticitytransport}
 \partial_t\omega+u\cdot\nabla\omega=0.
 \tag{2D Euler-vor}
\end{equation}
In this work, we give a Lagrangian construction of forward self-similar Euler flows, 
\begin{equation}\label{eq:selfsimilar_beginning}
 \omega(t,x)=t^{-1}\Omega(t^{-1/\alpha}x),\qquad
 u(t,x)=t^{1/\alpha-1}U(t^{-1/\alpha}x),\qquad t>0,\qquad 1<\alpha<2,
\end{equation}
starting from prescribed
homogeneous vorticity with rough angular dependence. Our construction uses the Lagrangian transport structure while
establishing the Euler equation through its velocity formulation,
which remains meaningful even when the ordinary vorticity flux
is not locally integrable.

The construction of self-similar solutions has a substantial
history in incompressible fluid equations. For two-dimensional
Euler, an important approach is to construct solutions with
spiraling vorticity structures. Elling initiated the construction
of algebraic spiral solutions, including flows with mixed-sign
vorticity \cite{ELLING2013JDE,EllingCMP2016}.
Shao--Wei--Zhang \cite{ShaoWeiZhang2025} extended this theory
to every rotational symmetry order at least two, while Choi
\cite{choi2025asymmetricselfsimilarspiralsolutions} removed
rotational symmetry. These constructions impose perturbative
or angular dominance conditions relative to a radial power-law
vortex. Homogeneous steady configurations were studied by
Choi--Coiculescu \cite{ChoiCoiculescu2026}, and related spirals
for generalized SQG were constructed by
Garc\'ia--G\'omez-Serrano \cite{GarciaGomezSerrano2026}.
More directly related to the present work, Choi \cite{Choi2026}
recently established self-similar Euler existence for arbitrary
homogeneous data with a $C^1$ angular velocity, without smallness
or sign assumptions. We compare the two results below.

The existence of self-similar solutions for large initial data
has also been studied extensively for Navier--Stokes equations.
For three-dimensional Navier--Stokes, early perturbative
constructions \cite{Giga1989,Cannone1996,KochTataru2001}
preceded the large-data existence theorem of
Jia--\v{S}ver\'ak \cite{JiaSverak_Invention}.
Subsequent work treated the half-space and discrete
self-similarity \cite{Tsai_halfspace_AnalPDE16,Tsai_discrete_AnnHP17},
as well as fractional dissipation
\cite{LAI2019981,Lai_TAMS21,LAI2024JFA}.
Recent two-dimensional constructions include the
Navier--Stokes results of Albritton--Guillod--Korobkov--Ren
and Gui--Liu--Xie
\cite{albritton2026forwardselfsimilarsolutions2d,gui2026forwardselfsimilarsolutionstwodimensional},
and the hypodissipative theory of Hou--Song and Choi \cite{HouSong2026,Choi2026}.

The connection with nonuniqueness is a principal motivation.
Yudovich-type uniqueness results
\cite{Yudovich1963,Yudovich1995}
contrast with the broader weak existence theory
\cite{DiPernaMajda1987,Delort1991,Schochet1995}.
For the unforced Euler equation on the two-dimensional
torus, Bru\`e--Colombo--Kumar \cite{BrueColomboKumar2026}
established nonuniqueness with vorticity bounded in
$L^p$ for some $p>1$, using convex integration
\cite{DeLellisSzekelyhidi2009}.
With forcing, nonuniqueness through self-similar
instability is known for every finite $p>1$
\cite{vishik2018I,vishik2018II,Instability_after_Vishik}.
Related developments include vanishing-viscosity
limits \cite{albritton2025vanishingviscositynonuniquesolutions}
and numerical evidence for distinct unforced
self-similar Euler evolutions from the same
homogeneous datum \cite{BressanShen2021}.
In three-dimensional Navier--Stokes,
self-similar bifurcation and instability have
also been studied in
\cite{JiaSverak2015JFA,Guillod2023,abc_forcing,
hypoNS_force23,hou2025nonuniquenesslerayhopfsolutionsunforced}.

Regularizing the Biot--Savart law provides a transport-based
approximation of Euler that is closely related to vortex blob
methods and the Euler-$\alpha$ or averaged Euler equations.
In these approximations, a vorticity variable is transported by
a regularized velocity rather than smoothed by diffusion.
DiPerna--Majda \cite{DiPernaMajda1987} studied concentration and
compactness in Euler regularizations, and Liu--Xin
\cite{LiuXin1995} proved convergence of vortex blob methods for
one-signed vortex-sheet data. Oliver--Shkoller
\cite{OliverShkoller2001} identified a particular vortex blob
regularization with averaged Euler and proved convergence for
bounded initial vorticity. Bardos--Linshiz--Titi
\cite{BardosLinshizTiti2010} established convergence of
Euler-$\alpha$ solutions for finite one-signed vorticity measures
with locally finite energy. Lagrangian properties of vortex approximations were also studied
by Ciampa--Crippa--Spirito \cite{CiampaCrippaSpirito2020}, who
proved that the weak Euler solutions obtained by their vortex
method are Lagrangian. In the vortex blob and averaged
Euler approximations just described, the smoothing length is
fixed in physical variables. Here we instead regularize in
self-similar variables, so that the physical smoothing length
follows the self-similar scaling. This preserves particle
transport while making the regularization compatible with
the profile fixed-point problem.

Our construction must resolve two related problems.
First, the vorticity profile must be consistent with the
Lagrangian transport generated by its own Biot--Savart velocity.
In particular, transporting the prescribed initial vorticity
forward to time one must recover the same profile.
Second, the resulting velocity must satisfy the Euler velocity
equation \eqref{eq:Euler}. The second assertion does not follow from a formal vorticity equation
when the velocity--vorticity product is not locally integrable.
For nonzero homogeneous data the total absolute vorticity is infinite,
so the construction must also control particles and nonlinear interactions
at arbitrarily large distances.

We identify $\mathbb S^1$ with $\R/(2\pi\mathbb Z)$. Fix a homogeneity exponent and its two critical exponents by
\begin{equation}\label{eq:exponents}
 1<\alpha<2,\qquad \qw=\frac2\alpha,\qquad
 \qu=\frac2{\alpha-1}=\frac{2\qw}{2-\qw}.
\end{equation}
We prescribe a real angular function $h\in L^{\qw}(\mathbb S^1)$ and set
\begin{equation}\label{eq:datum}
 H(x)=\rho^{-\alpha}h(\theta),\qquad
 x=\rho(\cos\theta,\sin\theta),\quad \rho>0.
\end{equation}
Its value at the origin is immaterial. We fix the counterclockwise
rotation and the corresponding rotated gradient by
\begin{equation}\label{eq:rotation}
 x^\perp=(-x_2,x_1),\qquad
 \nabla^\perp=(-\partial_2,\partial_1).
\end{equation}
The superscript $\perp$ always denotes this specific vector rotation.
Our vorticity and Biot--Savart conventions are
\begin{equation}\label{eq:kernel}
 \curl u=\partial_1u_2-\partial_2u_1,\qquad
 K(x)=\frac{x^\perp}{2\pi|x|^2},\quad x\ne0.
\end{equation}
For $1<\ell<\infty$, write $\ell'=\ell/(\ell-1)$ and use the Banach
Marcinkiewicz norm
\begin{equation}\label{eq:marcinkiewicz}
 \Mnorm{f}{\ell}
 =\sup_{0<|E|<\infty}|E|^{-1/\ell'}\int_E|f|\idiff x.
\end{equation}
Here\phantomsection\label{def:sets}
 $E$ ranges over measurable sets, $|E|$ is their measure, and the
modulus is Euclidean for vectors or matrices.
This\phantomsection\label{def:weaknorm}
norm actually defines the same space as the Lorentz space $L^{\ell,\infty}$ with the quasi-norm
$\|f\|_{L^{\ell,\infty}} := \sup_{\lambda>0}\lambda|\{|f|>\lambda\}|^{1/\ell}$. In fact, one has
\[
    \|f\|_{L^{\ell,\infty}}\leq \|f\|_{\calM^\ell}\leq \ell'\|f\|_{L^{\ell,\infty}}.
\]
In general, we recall the definition of Lorentz spaces $L^{p,q}$ on $\RR^2$, where $(p,q)\in [1,+\infty)\times[1,+\infty]$, characterized by the (quasi-)norms:
\[
\|f\|_{L^{p,q}} :=
\begin{cases}
\displaystyle
\ p^{\frac{1}{q}}
\left\|t|\{|f|>t\}|^{\frac{1}{p}}\right\|_{L^q(\RR^+,t^{-1}\idiff t)}
= p^{\frac{1}{q}}
\left(\int_0^{+\infty}t^{q-1}|\{|f|>t\}|^{\frac{q}{p}}
\idiff t\right)^{\frac{1}{q}},
& q<\infty,\\[1.2em]
\displaystyle
\ \left\|t|\{|f|>t\}|^{\frac{1}{p}}\right\|_{L^\infty(\RR^+)}
= \sup_{t>0}\,t|\{|f|>t\}|^{\frac{1}{p}},
& q=\infty.
\end{cases}
\]
By convention, $L^{\infty,\infty} = L^\infty$. We also recall some important properties of Lorentz spaces, which will be applied repeatedly below:
\begin{enumerate}
    \item (Lorentz H\"older inequalities) For $1\leq p,p_1,p_2<+\infty$, $1\leq q,q_1,q_2\leq+\infty$, $\frac{1}{p} = \frac{1}{p_1}+\frac{1}{p_2}$, and $\frac{1}{q} = \frac{1}{q_1}+\frac{1}{q_2}$, it holds that
    \begin{equation}
        \|fg\|_{L^{p,q}}\leq C_{p_1,p_2,q_1,q_2}\|f\|_{L^{p_1,q_1}}\|g\|_{L^{p_2,q_2}}.
    \end{equation}
    In particular, we will repeatedly use the pairing estimate
\begin{equation}\label{eq:lorentzholder}
 \int_{\R^2}|fg|\idiff x
 \leq C_\ell\Mnorm{f}{\ell}
 \|g\|_{L^{\ell',1}},
 \qquad 1<\ell<\infty.
\end{equation}
    \item (Embeddings) For a fixed $p$ and $1\leq q_1\leq q_2\leq +\infty$, we have
    \[
        L^{p,q_1}\subset L^{p,q_2}.
    \]
    Besides, for any $p\geq 1$, it holds that $\|f\|_{L^{p,p}} = \|f\|_{L^p}$.
    \item (Dualities) For $1<\ell<\infty$, the space $\calM^\ell$, with an equivalent norm, is the continuous dual of $L^{\ell',1}$, under the pairing
    \[
       \langle f,\varphi\rangle = \int_{\RR^2} f\varphi\idiff x.
    \]
    The space $L^{\ell',1}$ is separable, and bounded compactly supported functions are dense in it.
\end{enumerate}
The following standard Lorentz-space facts will also be used:
Convolution with kernel $|x|^{-1}$ maps
$L^{\ell,\infty}(\R^2)$ to $L^{2\ell/(2-\ell),\infty}$ for
$1<\ell<2$. Calder\'on--Zygmund operators and the Hardy--Littlewood
maximal operator are bounded on $L^{\ell,\infty}$ for $1<\ell<\infty$.
These statements follow from the usual strong estimates by real
interpolation. We refer interested readers to \cite{BennettSharpley1988,ONeil1963,Stein1970} for details. We will also repeatedly use the rearrangement consequence
\begin{equation}\label{eq:uniformintegrability}
 \int_E|f|^p\idiff x
 \leq C_{p,\ell}\Mnorm{f}{\ell}^{p}|E|^{1-p/\ell},
 \qquad 1\leq p<\ell,
\end{equation}
which can be derived from the Lorentz H\"older.
In the following, we set $M_H=\Mnorm{H}{\qw}$. 

We use the following differential and tensor conventions.
The vector Jacobian\phantomsection\label{def:jacobian}
has components $(Du)_{ij}=\partial_j u_i$.
The tensor product\phantomsection\label{def:tensorconventions}
has components $(u\otimes u)_{ij}=u_i u_j$.
Componentwise contraction and row divergence are given by
\[
 (u\otimes u):D\boldsymbol\zeta
 =\sum_{i,j=1}^2u_i u_j\partial_j\zeta_i,
 \qquad
 (\operatorname{div}(u\otimes u))_i
 =\sum_{j=1}^2\partial_j(u_i u_j),
\]
where $i,j\in\{1,2\}$ are Cartesian indices.
For the pressure, the Riesz transform
\phantomsection\label{def:riesz}
$\Riesz_i$ has Fourier multiplier
$-\mathrm{i}\xi_i/|\xi|$, where $\xi\in\R^2$
is the frequency variable and $\mathrm{i}$
is the imaginary unit.

\begin{theorem}
For\phantomsection\label{thm:main} every datum \eqref{eq:datum}, there is a real profile
$\Omega\in\mathcal M^{\qw}(\R^2)$ such that, with $U=K*\Omega$, the fields
\begin{equation}\label{eq:selfsimilar}
 \omega(t,x)=t^{-1}\Omega(t^{-1/\alpha}x),\qquad
 u(t,x)=t^{1/\alpha-1}U(t^{-1/\alpha}x),\qquad t>0,
\end{equation}
have the following properties.
\begin{enumerate}[label=\textup{(\roman*)},leftmargin=*]
\item The velocity $u(t)=K*\omega(t)$ is canonical,
and satisfies the critical bounds
\begin{equation}\label{eq:mainbounds}
 \Mnorm{\omega(t)}{\qw}=M_H,\qquad
 \Mnorm{u(t)}{\qu}+\Mnorm{Du(t)}{\qw}\leq C_\alpha M_H.
\end{equation}
The constant depends only on $\alpha$. The canonical pressure is defined by
\begin{equation}\label{eq:pressure}
 P(t)=\sum_{i,j=1}^2\Riesz_i\Riesz_j(u_i(t)u_j(t)),
 \qquad
 \Mnorm{P(t)}{\qu/2}\leq C_\alpha M_H^2.
\end{equation}
With this choice, $(u,P)$ satisfies the unforced
Euler momentum equation on the whole plane.
Precisely, for every vector test\phantomsection\label{def:testconventions}
$\boldsymbol\zeta\in C_c^\infty([0,\infty)\times\R^2;\R^2)$,
\begin{equation}\label{eq:velocity}
 \int_0^\infty\!\int_{\R^2}
 \bigl(u\cdot\partial_t\boldsymbol\zeta
 +(u\otimes u):D\boldsymbol\zeta
 +P\operatorname{div}\boldsymbol\zeta\bigr)\idiff x\idiff t
 +\int_{\R^2}(K*H)\cdot\boldsymbol\zeta(0)\idiff x=0.
\end{equation}
Also $\operatorname{div}u=0$ and $\curl u=\omega$ in distributions.
\item On every finite interval $[0,T]$, $u$ has a unique
measure-preserving regular Lagrangian flow for this prescribed velocity.
In two-time notation, $X_u(t,s,x)$ is the position at time $t$ of the
particle at $x$ at time $s$. Its trajectories are absolutely continuous
on the closed interval, its maps have almost-everywhere inverses
$X_u(s,t,\cdot)$, and
\begin{equation}\label{eq:exacttransport}
 \omega(t,x)=H(X_u(0,t,x)),\qquad \omega(0)=H.
\end{equation}
These identities hold almost everywhere in space at each fixed time.
\item For every $1\leq p<\qw$,
\begin{equation}\label{eq:traces}
 \omega\in C([0,\infty);L^p_{\loc}(\R^2)),\qquad
 u\in C([0,\infty);L^2_{\loc}(\R^2)),\qquad u(0)=K*H.
\end{equation}
Every\phantomsection\label{def:renormalization}
 bounded $\beta\in C^1(\R)$ with $\beta(0)=0$ and compactly
supported $\beta'$ satisfies
\begin{equation}\label{eq:renormalization}
 \int_0^\infty\!\int_{\R^2}\beta(\omega)
       (\partial_t\zeta+u\cdot\nabla\zeta)\idiff x\idiff t
 +\int_{\R^2}\beta(H)\zeta(0)\idiff x=0
\end{equation}
for every scalar\phantomsection\label{def:scalartests} $\zeta\in C_c^\infty([0,\infty)\times\R^2)$.
\item Let\phantomsection\label{def:angularparts}
 $h_+=\max(h,0)$ and $h_-=\max(-h,0)$. For $\lambda>0$ and
$t\geq0$ the signed distributions are exactly
\begin{equation}\label{eq:signeddistribution}
 |\{\omega(t)>\lambda\}|=D_H^+\lambda^{-\qw},\qquad
 |\{\omega(t)<-\lambda\}|=D_H^-\lambda^{-\qw},\qquad
 D_H^\pm=\frac12\int_0^{2\pi}h_\pm^{\qw}\idiff\theta.
\end{equation}
For\phantomsection\label{def:casimir}
 every Borel function $\mathcal F:\R\to\R$ such that $\mathcal F(H)\in L^1(\R^2)$,
one has $\mathcal F(\omega(t))\in L^1$ and
$\int \mathcal F(\omega(t))\idiff x=\int \mathcal F(H)\idiff x$.
\end{enumerate}
The same profile defines the solution for all positive times.
\end{theorem}

\begin{remark}
For nonzero $h$, the initial
vorticity belongs to no global strong $L^p$, $1\leq p\leq +\infty$.
The theorem makes no uniqueness or branch-selection assertion for Euler.
The uniqueness in part~(ii) concerns the flow of an already fixed velocity.
\end{remark}

\begin{remark}
 Our theorem establishes neither instability nor
Euler nonuniqueness. In particular, the problem
of constructing distinct unforced Euler solutions
with the same datum and vorticity in
$L^\infty_t(L^1_x\cap L^p_x)$ for arbitrarily
large finite $p$ remains separate.
Since $\qw=2/\alpha<2$, every nonzero datum
\eqref{eq:datum} fails to belong to
$L^p_{\loc}$ for $p\geq\qw$ because of its
singularity at the origin. Thus our construction
does not directly address the large-$p$
regularity class relevant to that problem.
\end{remark}
The profile $\Omega$ represents the vorticity at physical time $t=1$.
For a velocity profile $U$, consider the associated
self-similar physical velocity
\(
 b_U(t,x)=t^{1/\alpha-1}U(t^{-1/\alpha}x).
\)
Whenever its regular Lagrangian flow $X_{b_U}$ is defined,
we denote by
\[
 L_U=X_{b_U}(0,1,\cdot)
\] 
the backward endpoint map that assigns to a particle at time one
its position at time zero.
For a self-similar solution transporting the initial vorticity $H$,
the profile must agree with the transport of $H$ from time $t=0$
to time $t=1$. This leads to the fixed-point relation
\(
 \Omega=H\circ L_{K*\Omega},
\)
which we seek to solve through the successive operations
for functions $w \in \mathcal{M}^{\qw}$:
\[
 w\longmapsto K*w
 \longmapsto L_{K*w}
 \longmapsto H\circ L_{K*w}.
\]
The existence and stability of the endpoint map
will be established in Section~\ref{sec:formulation}, with its
precise construction in \eqref{eq:endpoint}. Its measure preservation makes the critical
norm of every output equal to that of $H$. Stability of the flow must
hold through time zero, and composition must remain continuous for
merely measurable angular data. These are the first obstacles to the
fixed-point construction.

The remaining obstacle is nonlinear. Indeed,
\begin{equation}\label{eq:fluxthreshold}
 \frac1{\qw}+\frac1{\qu}=\alpha-\frac12.
\end{equation}
For $\alpha\geq 3/2$, these bounds do not imply $u\omega\in L^1_{\mathrm{loc}}$. In fact, the stationary example in Appendix \ref{app:comparison} shows that this product can fail to be locally integrable, including at $\alpha=3/2$.
We therefore establish the velocity
identity \eqref{eq:velocity} rather than relying on the ordinary
weak-vorticity formulation. Separately, we prove the renormalized
transport identity \eqref{eq:renormalization} for bounded functions
$\beta(\omega)$, with $\beta$ satisfying the assumptions stated
in Theorem~\ref{thm:main}. To establish these identities, we retain a sequence of fixed
points with mollified profile velocities. In physical variables,
the corresponding advecting fields satisfy
\[
 b_n(t)=\eta_{\delta_n(t)}*(K*\omega_n(t)),\qquad
 \delta_n(t)=\varepsilon_n t^{1/\alpha},
\]
where $\eta_\delta$ is a radial mollifier; see
\eqref{eq:physicalregularization}. Thus a fixed smoothing radius
in profile variables corresponds to a time-dependent physical
radius that vanishes at time zero. Each approximate vorticity
is transported exactly from $H$ by $b_n$, and its flux
$b_n\omega_n$ is locally integrable through time zero, yielding
\eqref{eq:approxtransport}. Regularization is used to justify
this flux; the unregularized endpoint flow itself already exists.
Flow stability gives the strong local vorticity convergence in
\eqref{eq:scalarlimit}. Symmetrization then permits passage to
the nonlinear limit without forming the ordinary limiting flux
$u\omega$. The velocity stress identity and a global pressure
estimate recover the whole-plane Euler velocity equation.
Thus the theorem is proved for the fixed point selected by
this construction, rather than inferred from the profile
identity alone.

Choi's Theorem~1.2 in version~2 of \cite{Choi2026}
and Theorem~\ref{thm:main} cover the same homogeneity range.
His\phantomsection\label{def:choiparameter}
velocity parameter is $a=\alpha-1$.
Both results give global self-similar weak Euler flows without
smallness, sign, or rotational symmetry assumptions, together
with critical velocity bounds and strong local $L^2$ velocity
continuity. Choi's stated theorem assumes a $C^1$ angular
velocity, whose inclusion in our datum class is identified in
Proposition~\ref{prop:angularvelocity}. His quantitative
estimates also yield basic weak-Euler existence for the rough
data considered here through angular approximation, as
explained in Appendix~\ref{app:comparison}.
Choi has pointed out to us that his vanishing-viscosity limits
are also renormalized, by an argument analogous to that of
Crippa--Spirito \cite{CrippaSpirito2015}.\footnote{Personal
communication. This property is not stated in version~2 of
\cite{Choi2026}.}
Under the corresponding Sobolev and growth bounds,
DiPerna--Lions theory \cite[Theorem~III.2]{DiPernaLions1989}
then identifies these vorticities with transport by their
measure-preserving regular Lagrangian flows. The Lagrangian
character of the limiting solution is therefore not the
distinction between the two approaches.

The distinction lies in the construction. Choi first constructs
profiles with scale-compatible hypodissipation and obtains
Euler solutions through velocity compactness in a vanishing
viscosity process. Our fixed-point problem is instead formulated
directly through backward particle transport,
$\Omega=H\circ L_{K*\Omega}$. Its regularization mollifies the
advecting velocity while preserving exact transport of the
prescribed vorticity. Stability of the endpoint maps gives
strong local vorticity convergence, while a separate symmetrized
argument establishes the Euler momentum equation even when
the ordinary vorticity flux is not locally integrable.
We do not assert that the profiles selected by the two
constructions coincide.

The flow theory originates in DiPerna--Lions
\cite{DiPernaLions1989}. Our quantitative endpoint argument
uses the logarithmic-distance method of Crippa--De Lellis
\cite{CrippaDeLellis2008}, with a global weak-Lorentz
displacement estimate replacing pointwise growth control.
For the nonlinear limit, we use the standard symmetrized
weak-vorticity formulation used by Delort \cite{Delort1991}
and developed further by Schochet \cite{Schochet1995}.
We verify its continuity and its equivalence to the velocity
stress in the present infinite-mass critical class.

The remainder of the paper is organized as follows.
Section~\ref{sec:formulation} formulates the transported-profile problem through the endpoint map, constructs regularized fixed points, and extracts a convergent sequence. Section~\ref{sec:realization} passes to the limit along this sequence and establishes the Euler equation, the initial traces, and the remaining properties.
Appendix~\ref{app:flow} supplies the closed-time flow proof.
Appendix~\ref{app:comparison} identifies the angular velocity, compares
datum classes, and gives a stationary example demonstrating failure
of the ordinary flux.

\section{The Lagrangian fixed-point formulation}\label{sec:formulation}

A self-similar transported vorticity is determined by its value
$\Omega$ at time $t=1$. Following its particles back to the prescribed
datum suggests the profile relation
\begin{equation}\label{eq:profileproblem}
 \Omega=H\circ L_{K*\Omega}.
\end{equation}
Here and below, $L_U$ denotes the backward endpoint map of the
self-similar velocity associated with $U$, constructed rigorously
in \eqref{eq:endpoint}.
The main result of this section is the construction of a solution of
\eqref{eq:profileproblem}, together with a regularized sequence that
will be retained for the later Euler realization.
Fix a nonnegative radial function
$\eta\in C_c^\infty(B_1)$ with unit integral, and write
$\eta_\varepsilon(x)=\varepsilon^{-2}\eta(x/\varepsilon)$.

\begin{theorem}[Lagrangian fixed point and regularized approximation]
\label{thm:lagrangianfixedpoint}
There exist a sequence $\varepsilon_n\downarrow0$, real profiles
$\Omega_n,\Omega\in\mathcal M^{\qw}(\R^2)$, and profile velocities
\begin{equation}\label{eq:chapterprofiles}
 U_n=\eta_{\varepsilon_n}*(K*\Omega_n),
 \qquad
 U=K*\Omega,
\end{equation}
such that the regularized and limiting fixed-point identities are
\begin{equation}\label{eq:chapterfixedpoints}
 \Omega_n=H\circ L_{U_n},
 \qquad
 \Omega=H\circ L_U.
\end{equation}
Each profile is equimeasurable with $H$. In particular, the exact
critical norms and the uniform velocity bounds are
\begin{equation}\label{eq:chapterbounds}
 \Mnorm{\Omega_n}{\qw}
 =\Mnorm{\Omega}{\qw}
 =M_H,
 \qquad
 \Mnorm{U_n}{\qu}
 +\Mnorm{DU_n}{\qw}
 +\Mnorm{U}{\qu}
 +\Mnorm{DU}{\qw}
 \leq C_\alpha M_H.
\end{equation}
The convergence $\Omega_n\to\Omega$ is strong in
$L^p_{\loc}(\R^2)$ for every $1\leq p<\qw$ and weak-* against
$L^{\qwp,1}$. The convergence $U_n\to U$ is strong in
$L^1_{\loc}(\R^2)$. The endpoint maps $L_{U_n}$ and their inverses
converge locally in measure to $L_U$ and $L_U^{-1}$, respectively.
\end{theorem}

The remainder of this section proves
Theorem~\ref{thm:lagrangianfixedpoint} by the
Schauder--Tychonoff fixed-point theorem. The fixed-point domain is the
critical ball
\begin{equation}\label{eq:ball}
 \mathcal C_H
 =
 \left\{
 w\in\mathcal M^{\qw}(\R^2;\R):
 \Mnorm{w}{\qw}\leq M_H
 \right\},
\end{equation}
equipped with the weak-* topology
\(
 \sigma\bigl(\mathcal M^{\qw},L^{\qwp,1}\bigr).
\)
For $0<\varepsilon\leq1$, define
\begin{equation}\label{eq:fixedmap}
 U_{\varepsilon,w}
 =\eta_\varepsilon*(K*w),
 \qquad
 \mathcal T_\varepsilon(w)
 =H\circ L_{U_{\varepsilon,w}}.
\end{equation}
The Schauder--Tychonoff argument requires a
weak-* compact convex domain and a continuous
self-map. To establish continuity of
$\mathcal T_\varepsilon$, we combine canonical
Biot--Savart compactness, stability of the backward
endpoint flow, and continuity of composition
with the rough datum. Measure preservation
supplies the self-map property.

The use of the regularized maps $\mathcal T_\varepsilon$ serves an
additional purpose. The unregularized transported-profile map is also continuous. We nevertheless construct regularized fixed points because, for each fixed regularization, the transported vorticity has a locally integrable flux and satisfies an ordinary weak transport identity. In Section \ref{sec:realization} we symmetrize these identities and pass to the limit to prove the Euler velocity equation.

\subsection{Critical fields and the endpoint map}\label{subsec:critical}
We begin with the first two steps in the decomposition of
$\mathcal T_\varepsilon$. We first identify the critical scalar class
and establish the estimates governing the passage from a trial
vorticity to its canonical velocity. These bounds then provide the
functional setting for the associated endpoint flow.
We\phantomsection\label{def:balls}
 write $B_R$ for the centered open ball of radius $R>0$ and use
$L>R$ for a larger confinement radius.

\begin{lemma}[Exact distribution]
For\phantomsection\label{lem:distribution} every measurable real $h$, the following identity holds (recall that $H:= \rho^{-\alpha}h(\theta)$)
\begin{equation}\label{eq:distribution}
 |\{|H|>\lambda\}|=D_H\lambda^{-\qw},\qquad
 D_H=\frac12\int_0^{2\pi}|h|^{\qw}\idiff\theta.
\end{equation}
Consequently $H\in\mathcal M^{\qw}$ if and only if
$h\in L^{\qw}(\mathbb S^1)$, and
\begin{equation}\label{eq:exactnorm}
 \weaknorm{H}{\qw}=2^{-1/\qw}\|h\|_{L^{\qw}},\qquad
 M_H=\qwp\,2^{-1/\qw}\|h\|_{L^{\qw}}.
\end{equation}
Moreover, for $1\leq p<\qw$ and the centered open ball $B_R$,
\begin{equation}\label{eq:localdatum}
 \int_{B_R}|H|^p\idiff x
 =\frac{R^{2-\alpha p}}{2-\alpha p}
       \int_0^{2\pi}|h|^p\idiff\theta.
\end{equation}
\end{lemma}
\begin{proof}
For \eqref{eq:distribution},
\begin{align}
     |\{|H|>\lambda\}| =\int_{0}^{2\pi}\int_{0}^{\frac{|h(\theta)|^{1/\alpha}}{\lambda^{1/\alpha}}}\rho\idiff\rho\idiff\theta = \frac{\lambda^{-2/\alpha}}{2}\int_{0}^{2\pi}|h(\theta)|^{2/\alpha}\idiff\theta = D_H\lambda^{-\qw}.
\end{align}
The first identity of \eqref{eq:exactnorm} follows directly from \eqref{eq:distribution} and the definition of Lorentz norm $\|\cdot\|_{L^{q_\omega,\infty}}$. For\phantomsection\label{def:rearrangement}
 $D_H>0$ finite, the decreasing rearrangement, with positive
measure coordinates $\mu,\nu$, is
\begin{equation}\label{eq:rearrangement}
 H^*(\mu)=(D_H/\mu)^{1/\qw},\qquad \mu>0.
\end{equation}
The supremum of
$\int_E|H|$ at measure $|E|=\mu$ equals
$\int_0^\mu H^*(\nu)\idiff\nu
=\qwp D_H^{1/\qw}\mu^{1/\qwp}$, proving the second identity in
\eqref{eq:exactnorm}. The zero case is immediate. Direct radial
integration proves \eqref{eq:localdatum}. The same superlevel computation
with $h_+$ and $h_-$ proves the datum version of
\eqref{eq:signeddistribution}. For nonzero $h$, the radial integral of
$\rho^{1-\alpha p}$ cannot converge both at zero and at infinity.
\end{proof}

Thus the strong angular $L^{\qw}$ assumption is precisely the
condition for the homogeneous extension $H$ to belong to
$\mathcal M^{\qw}(\R^2)$.

The preceding distribution computation identifies the domain of the
first step in $\mathcal T_\varepsilon$. We now establish the bounds and
compactness needed for the map from a trial vorticity to its profile
velocity.
\begin{lemma}[Canonical Biot--Savart compactness]
For\phantomsection\label{lem:velocity} real $w\in\mathcal M^{\qw}$, the integral $K*w$ is absolutely
convergent almost everywhere. It satisfies
\begin{equation}\label{eq:operators}
 \Mnorm{K*w}{\qu}+\Mnorm{D(K*w)}{\qw}
 \leq C_\alpha\Mnorm{w}{\qw},\qquad
 \operatorname{div}(K*w)=0,\quad \curl(K*w)=w.
\end{equation}
Its absolute tail obeys
\begin{equation}\label{eq:tail}
 \int_{|y|>R}\frac{|w(y)|}{|y|}\idiff y
 \leq C_\alpha\Mnorm{w}{\qw}R^{1-\alpha}.
\end{equation}
If $w_n$ is bounded in $\mathcal M^{\qw}$ and converges weak-* to $w$
against $L^{\qwp,1}$, then
\begin{equation}\label{eq:velocitycompact}
 K*w_n\longrightarrow K*w\quad\hbox{strongly in }L^p_{\loc},
 \qquad 1\leq p<\qu.
\end{equation}
\end{lemma}
\begin{proof}
Here $\qwp=2/(2-\alpha)>2$. The function
$|y|^{-1}\ind_{\{|y|>R\}}$ belongs to $L^{\qwp,1}$, with norm
$C_\alpha R^{1-\alpha}$, by dilation and the distribution of $|y|^{-1}$
outside the unit disk. Applying \eqref{eq:lorentzholder}
with $\ell=\qw$ gives the tail estimate \eqref{eq:tail}.
The standard convolution estimate
\[
\left\|
\int_{\mathbb R^2}\frac{|w(y)|}{|\cdot-y|}\,dy
\right\|_{\mathcal M^{q_u}}
\leq C_\alpha\|w\|_{\mathcal M^{q_\omega}}
\]
shows that the integral on the left is finite almost everywhere. Since \(|K(z)|=(2\pi|z|)^{-1}\), this proves absolute convergence of \(K*w\) almost everywhere and gives the velocity estimate in \eqref{eq:operators}. The derivative is a matrix of singular integrals of $w$, including
the usual distributional local term. Their bounds prove the gradient
estimate. Testing the fundamental-solution identities for $K$ against
compact smooth functions, justified by the same local potential and
tail bounds, gives zero divergence and curl $w$ on the entire plane.

The operator bounds \eqref{eq:operators}, together with
\eqref{eq:uniformintegrability}, give uniform
$W^{1,\ell}$ bounds for the velocities on every ball,
for each $1<\ell<\qw$. Rellich compactness then gives
local strong convergence along subsequences for exponents
below $2\ell/(2-\ell)$.
To identify the limit, let $\boldsymbol\zeta$ be a compact smooth vector
test. Its adjoint kernel
$y\mapsto\int K(x-y)\cdot\boldsymbol\zeta(x)\idiff x$
is locally bounded and is $\mathcal O(|y|^{-1})$ at infinity, hence belongs to
$L^{\qwp,1}$. Weak-* convergence therefore identifies every subsequential
limit as $K*w$. Letting $\ell$ increase to $\qw$ proves
\eqref{eq:velocitycompact} for the whole sequence.
\end{proof}

The preceding lemma completes the first step in the continuity argument
for $\mathcal T_\varepsilon$. We next turn to the second step, the
dependence of the endpoint map on the profile velocity. The canonical
velocity estimates \eqref{eq:operators} give precisely the bounds required for this flow
construction. The full closed-time argument is deferred to
Appendix~\ref{app:flow}, while the properties used in the fixed-point
argument are recorded here.
\begin{proposition}[Flows with global weak bounds]
Let\phantomsection\label{prop:flow} $b$ be a measurable divergence-free field on $[0,T]\times\R^2$ with
\begin{equation}\label{eq:flowbounds}
 \esssup_{0<t<T}\Mnorm{b(t)}{\qu}\leq C_v,\qquad
 \esssup_{0<t<T}\Mnorm{Db(t)}{\qw}\leq C_\nabla.
\end{equation}
It has an almost-everywhere invertible measure-preserving two-time
regular Lagrangian flow $X_b(t,s,x)$, unique among regular flows, with
absolutely continuous trajectories on $[0,T]$.
For a particle located at $x$ at the starting time $s$, these statements include
\begin{equation}\label{eq:rlf}
 \begin{gathered}
 X_b(t,s,x)=x+\int_s^t b(t',X_b(t',s,x))\idiff t',\\
 (X_b(t,s,\cdot))_\#\mathcal L^2=\mathcal L^2,
 \end{gathered}
\end{equation}
where $\mathcal L^2$ is planar Lebesgue measure and the subscript
$\#$ denotes pushforward. At physical time zero we write the forward
trajectory as $X_b(t,0,x_0)$. We use $X_b(0,t,x)$ for the time-zero label of
the particle located at $x$ at time $t$.
Let\phantomsection\label{def:intervalidentity}
 $I_{s,t}$ be the closed interval between $s$ and $t$, and let
$\Idmap(x)=x$ be the identity map. Then
\begin{equation}\label{eq:displacement}
 \Mnorm{\sup_{t'\in I_{s,t}}
       |X_b(t',s,\cdot)-\Idmap|}{\qu}
 \leq C_v|t-s|.
\end{equation}
If $b^{(n)}\to b$ in local spacetime $L^1$, with common bounds \eqref{eq:flowbounds}, then for every fixed starting time $s$, every $R>0$, and every $\gamma>0$,
\[
|\{
x\in B_R:
\sup_{s\leq t\leq T}
|X_{b^{(n)}}(t,s,x)-X_b(t,s,x)|>\gamma
\}|\to0.
\]
The analogous statement holds for backward trajectories, with the supremum over $0\leq t\leq s$.
\end{proposition}


Here a regular Lagrangian flow means an almost-everywhere solution of
the integral ODE \eqref{eq:rlf} with a uniform bound on the pushforward density of
Lebesgue measure. In the proposition that density is exactly one.

The bounds \eqref{eq:flowbounds} also verify the standard
hypotheses of Sobolev regular Lagrangian flow theory,
which require local spatial Sobolev regularity, control
of the divergence, and an integrable global growth condition.
Thus, for every $1<\ell<\qw$, the field $b$ belongs to
$L^1(0,T;W^{1,\ell}_{\loc})$, while its divergence vanishes.
To verify the growth condition, we use the decomposition
\begin{equation}\label{eq:growth}
 b=b\ind_{\{|b|>1\}}+b\ind_{\{|b|\leq1\}},\qquad
 \|b\ind_{\{|b|>1\}}\|_1\leq C_{\qu}C_v^{\qu},\qquad
 \|b\ind_{\{|b|\leq1\}}\|_\infty\leq1.
\end{equation}
The estimates in \eqref{eq:growth} imply the required growth
condition in $L^1(0,T;L^1+L^\infty)$ for the weighted field
$b/(1+|x|)$. No pointwise linear-growth bound is inferred
from the weak norm.
From \eqref{eq:displacement} and Appendix~\ref{app:flow}, for
$L>R$, the measure of labels in $B_R$ whose paths leave $B_L$ is at most
\begin{equation}\label{eq:escape}
|\{
x\in B_R:  
\sup_{t'\in I_{s,t}}|X_b(t',s,x)|>L
\}|
\leq
\left(\frac{C_v|t-s|}{L-R}\right)^{\qu}.
\end{equation}
The same bound holds backwards and therefore controls preimages of
bounded sets. This estimate bounds the area of starting points whose trajectories leave $B_L$. It does not assert that every trajectory remains inside $B_L$.

Apply this result to any divergence-free profile velocity
$U\in\mathcal M^{\qu}$ with $DU\in\mathcal M^{\qw}$. Define its physical drift and time-reversed drift by
\begin{equation}\label{eq:incoming}
 b_U(t,x)=t^{1/\alpha-1}U(t^{-1/\alpha}x),\qquad
 V_U(\tau,x)=-b_U(1-\tau,x),\quad 0\leq\tau<1.
\end{equation}
Values at the single exceptional times are irrelevant. Exact scaling
gives
\begin{equation}\label{eq:criticalscaling}
 \Mnorm{b_U(t)}{\qu}=\Mnorm{U}{\qu},\qquad
 \Mnorm{Db_U(t)}{\qw}=\Mnorm{DU}{\qw}.
\end{equation}
Indeed the scaling powers are
$1/\alpha-1+2/(\alpha\qu)=0$ and $-1+2/(\alpha\qw)=0$.
Proposition~\ref{prop:flow} applies on every finite physical interval,
and also to $V_U$ on the actual closed interval $[0,1]$. Denote its
incoming\phantomsection\label{def:incomingflow}
 flow at incoming times $\tau,\tau_0$ by
$Z_U(\tau,\tau_0,x)=X_{V_U}(\tau,\tau_0,x)$, and set
\begin{equation}\label{eq:endpoint}
 L_U=Z_U(1,0,\cdot)=X_{b_U}(0,1,\cdot).
\end{equation}
In particular $L_U$ and its inverse are genuine endpoint maps.

\subsection{The transported-profile map}\label{subsec:flow}
We now turn to the third step in the continuity argument for
$\mathcal T_\varepsilon$, namely composition of the rough datum with
the endpoint map. Since the endpoint map is invertible and measure
preserving, this composition preserves the full distribution of $H$.
It remains to prove its continuity with respect to the underlying
velocity.
For rough data, this requires both local convergence of the endpoint
maps and control of labels whose images escape to large distances.
The following lemma isolates the second issue.

\begin{lemma}[Composition with rough data]
Suppose\phantomsection\label{lem:composition} $F_n,F$ are measure-preserving maps of the plane, $F_n\to F$
locally in measure, and for every $R$ their images of $B_R$ are uniformly
tight in the following sense.
\begin{equation}\label{eq:maptightness}
 \lim_{L\to\infty}\sup_n
 |\{x\in B_R:|F_n(x)|>L\}|=0,
\end{equation}
with the same property for $F$. For every fixed
$f\in\mathcal M^{\qw}$,
\begin{equation}\label{eq:composition}
 f\circ F_n\longrightarrow f\circ F
 \quad\hbox{in }L^p_{\loc},\qquad 1\leq p<\qw,
\end{equation}
and weak-* against $L^{\qwp,1}$.
\end{lemma}

\begin{proof}
Fix $R>0$. We first establish the local convergence in measure
of $f\circ F_n$ to $f\circ F$.
For any $L>0$, the estimate \eqref{eq:uniformintegrability}
ensures that $f\in L^1(B_L)$. Given $\delta>0$, choose a function
$g\in C_c^\infty(\R^2)$ approximating $f$ in measure on $B_L$,
and set
\(
 A_{L,\delta}
 =\{y\in B_L:|f(y)-g(y)|>\delta\}.
\)
The measure of $A_{L,\delta}$ can be made arbitrarily small.
Since $F_n$ and $F$ preserve Lebesgue measure,
\(
 |F_n^{-1}(A_{L,\delta})|
 =|F^{-1}(A_{L,\delta})|
 =|A_{L,\delta}|.
\)
Thus the triangle inequality gives
\[
\begin{aligned}
\bigl|\{x\in B_R:
 |f(F_n(x))-f(F(x))|>3\delta\}\bigr|
\leq&
 |\{x\in B_R:|F_n(x)|>L\}|+
 |\{x\in B_R:|F(x)|>L\}|\\
&+2|A_{L,\delta}|+
 |\{x\in B_R:
 |g(F_n(x))-g(F(x))|>\delta\}|.
\end{aligned}
\]
The first two terms are uniformly small for sufficiently large
$L$ by \eqref{eq:maptightness}. For each fixed $L$, the third
term can be made arbitrarily small by the choice of $g$.
Since $g$ is uniformly continuous and $F_n\to F$ locally
in measure, the last term tends to zero as $n\to\infty$.
This proves the local convergence in measure of
$f\circ F_n$ to $f\circ F$.

We next upgrade this convergence to strong local $L^p$
convergence. Measure preservation gives exact
equimeasurability and hence
\(
 \Mnorm{f\circ F_n}{\qw}
 =\Mnorm{f\circ F}{\qw}
 =\Mnorm{f}{\qw}.
\)
For every measurable set $E$ of finite measure,
\eqref{eq:uniformintegrability} therefore yields
\[
 \sup_n\int_E|f\circ F_n|^p\idiff x
 \leq C_{p,\qw}\Mnorm{f}{\qw}^p
 |E|^{1-p/\qw}.
\]
The same estimate holds for $f\circ F$. Since $p<\qw$,
these estimates give uniform integrability of the
$p$th powers on bounded sets. Together with the local
convergence in measure established above, Vitali's
theorem proves the strong local $L^p$ convergence
in \eqref{eq:composition}.

Finally, we prove weak-* convergence against
$L^{\qwp,1}$. Let $\varphi\in L^{\qwp,1}$ and choose
$\varphi_j\in C_c^\infty(\R^2)$ converging to
$\varphi$ in $L^{\qwp,1}$. The strong local $L^1$
convergence gives
\[
 \int_{\R^2}(f\circ F_n-f\circ F)\varphi_j
 \idiff x\longrightarrow0
\]
for every fixed $j$. By the Lorentz--H\"older estimate
\eqref{eq:lorentzholder} and exact equimeasurability,
\[
\begin{aligned}
\left|
 \int_{\R^2}(f\circ F_n-f\circ F)
 (\varphi-\varphi_j)\idiff x
 \right|\leq
 2C_{\qw}\Mnorm{f}{\qw}
 \|\varphi-\varphi_j\|_{L^{\qwp,1}},
\end{aligned}
\]
uniformly in $n$. Letting $n\to\infty$ first and then
$j\to\infty$ proves the weak-* conclusion.
\end{proof}

The next lemma applies the general flow stability result to the
self-similar velocities and provides the endpoint continuity
required for the fixed-point construction.

\begin{lemma}[Stability of the endpoint map]
\label{lem:endpointstability}
Let $U_n,U$ be divergence-free profile velocities with
uniformly bounded $\mathcal M^{\qu}$ velocity norms and
$\mathcal M^{\qw}$ gradient norms. Suppose that $U_n\to U$ strongly in $L^1_{\loc}(\R^2)$.
Then the convergences $L_{U_n}\to L_U$ and
$L_{U_n}^{-1}\to L_U^{-1}$ hold locally in measure.
\end{lemma}

\begin{proof}
By the scaling identities \eqref{eq:criticalscaling}, the
associated physical velocities satisfy the common bounds
\eqref{eq:flowbounds}. We verify their local spacetime
convergence through time zero.
Fix $R,T>0$ and a cutoff $t_*\in(0,T)$. On $[t_*,T]$,
a change of variables gives
\[
\int_{t_*}^{T}\int_{B_R}
 |b_{U_n}(t,x)-b_U(t,x)|\idiff x\idiff t\leq
\left(\int_{t_*}^{T}t^{3/\alpha-1}\idiff t\right)
\|U_n-U\|_{L^1\big(B_{Rt_*^{-1/\alpha}}\big)},
\]
which tends to zero for each fixed $t_*$. On $[0,t_*]$,
the common weak velocity bound gives
\[
 \int_0^{t_*}\int_{B_R}
 |b_{U_n}(t,x)-b_U(t,x)|\idiff x\idiff t
 \leq 2C_v t_*|B_R|^{1-1/\qu}.
\]
Letting $n\to\infty$ and then $t_*\downarrow0$ proves
local spacetime $L^1$ convergence on $[0,T]$.

The stability conclusion of Proposition~\ref{prop:flow}
now gives convergence of the forward and backward flows.
Applying the same argument to the reversed-time drift in
\eqref{eq:incoming} gives convergence of the incoming
endpoint maps. The identification \eqref{eq:endpoint}
then proves the asserted convergence of $L_{U_n}$ and
their inverses.
\end{proof}

The escape estimate \eqref{eq:escape}, applied to the
forward and backward flows, verifies the tightness
condition \eqref{eq:maptightness} for the endpoint maps.
Combining this property with Lemmas~\ref{lem:endpointstability}
and \ref{lem:composition}, we obtain the strong local
$L^p$ convergence and global weak-* convergence of
$H\circ L_{U_n}$ to $H\circ L_U$.

Consequently, the map $w\mapsto H\circ L_{K*w}$
is well defined on every bounded critical scalar ball,
preserves the distribution of $H$, and is continuous from
weak-* scalar convergence to strong subcritical local
convergence of its outputs. This follows from the velocity
compactness \eqref{eq:velocitycompact}, the endpoint stability
in Lemma~\ref{lem:endpointstability}, and the composition
continuity in Lemma~\ref{lem:composition}. The same results
also give weak-* continuity against $L^{\qwp,1}$.
For a canonically self-similar field already known to transport $H$,
evaluation at time one gives \eqref{eq:profileproblem}.
Conversely, a solution of this relation gives exact self-similar
transport by the covariance proved in Lemma~\ref{lem:covariance}.
Note that this statement only resolves transport. To obtain the Euler velocity
equation, we construct a fixed point together with a suitable
regularized sequence.

\subsection{Existence of the fixed point}\label{subsec:fixed}
We now complete the fixed-point construction leading to
Theorem~\ref{thm:lagrangianfixedpoint}. We first record the topological
properties of the fixed-point domain $\mathcal C_H$. The domain itself
is the full critical norm ball rather than an equimeasurable class,
which preserves the convexity required by the Schauder--Tychonoff
theorem. By contrast, every output of the transported-profile map is
equimeasurable with $H$, a property that will give the self-map
condition.

\begin{lemma}[The compact domain]
The\phantomsection\label{lem:ball} set $\mathcal C_H$ defined in
\eqref{eq:ball} is nonempty, convex, compact, and metrizable for the
weak-* topology induced by $L^{\qwp,1}$.
\end{lemma}
\begin{proof}
It lies in a bounded ball of the dual of $L^{\qwp,1}$.
For\phantomsection\label{def:dualtests} a fixed set $E$ of finite measure,
$\int_E|w|$ is the supremum of the continuous pairings
$\int w\varpi$ over measurable $\varpi$ supported in $E$ with $|\varpi|\leq1$.
Such $\varpi$ belong to $L^{\qwp,1}$. Every constraint defining $\mathcal C_H$ is therefore weak-* closed.
Banach--Alaoglu gives compactness, and separability of
$L^{\qwp,1}$ gives metrizability on this bounded set.
Convexity follows from the norm inequality, while the zero
function proves nonemptiness.
\end{proof}

With the compactness of $\mathcal C_H$ established,
it remains to verify continuity of the regularized
maps and then remove the regularization along
a selected sequence. We then apply the Schauder--Tychonoff theorem and remove
the regularization through the endpoint continuity established in
Subsection~\ref{subsec:flow}. The resulting fixed-point sequence is retained
for the later Euler realization.

\begin{proof}[Proof of Theorem~\ref{thm:lagrangianfixedpoint}]
Fix $0<\varepsilon\leq1$. We first verify that the map
$\mathcal T_\varepsilon$ defined in \eqref{eq:fixedmap} is a
continuous self-map of $\mathcal C_H$.
Convolution by $\eta_\varepsilon$ is a contraction in every
Marcinkiewicz norm used here. Indeed, for every measurable set $E$ of
finite measure and every $1<\ell<\infty$, Fubini's theorem and
\eqref{eq:marcinkiewicz} give the estimate
\begin{equation}
 \int_E |\eta_\varepsilon*f(x)|\idiff x
 \leq
 \int_{\R^2}\eta_\varepsilon(y)
       \int_{E-y}|f(z)|\idiff z\idiff y
 \leq
 \Mnorm{f}{\ell}|E|^{1/\ell'}.
\end{equation}
Taking the supremum over $E$ gives the contraction estimate
\begin{equation}\label{eq:mollificationcontraction}
 \Mnorm{\eta_\varepsilon*f}{\ell}
 \leq \Mnorm{f}{\ell},
 \qquad 1<\ell<\infty.
\end{equation}
Since convolution commutes with distributional derivatives,
mollification also preserves the divergence-free condition and gives
\[
 D U_{\varepsilon,w}
 =
 \eta_\varepsilon*D(K*w).
\]
Lemma~\ref{lem:velocity} therefore yields the uniform critical bounds
\begin{equation}
 \Mnorm{U_{\varepsilon,w}}{\qu}
 +
 \Mnorm{D U_{\varepsilon,w}}{\qw}
 \leq
 C_\alpha\Mnorm{w}{\qw}
 \leq
 C_\alpha M_H,
 \qquad w\in\mathcal C_H.
\end{equation}
Thus $U_{\varepsilon,w}$ satisfies the hypotheses needed to
construct the endpoint map $L_{U_{\varepsilon,w}}$, uniformly in
$\varepsilon$ and $w\in\mathcal C_H$.

The self-map property follows from measure preservation of the endpoint
flow. For every $w\in\mathcal C_H$, the function
$\mathcal T_\varepsilon(w)$ is equimeasurable with $H$. Consequently,
its critical Marcinkiewicz norm is $M_H$, and therefore
\begin{equation}\label{eq:selfmap}
 \mathcal T_\varepsilon(\mathcal C_H)\subseteq\mathcal C_H.
\end{equation}

We next verify continuity in the weak-* topology of
$\mathcal C_H$. Suppose that
$w_n\rightharpoonup^*w$ against $L^{\qwp,1}$. 
By Lemma~\ref{lem:velocity}, the corresponding canonical velocities
converge strongly locally. Convolution with the fixed mollifier
$\eta_\varepsilon$ gives the required local convergence of
$U_{\varepsilon,w_n}$ to $U_{\varepsilon,w}$. The stability of the
endpoint flows in Lemma~\ref{lem:endpointstability} then gives local convergence in measure of
$L_{U_{\varepsilon,w_n}}$ to $L_{U_{\varepsilon,w}}$ and of their
inverses. Lemma~\ref{lem:composition} consequently gives
\[
 \mathcal T_\varepsilon(w_n)
 \longrightarrow
 \mathcal T_\varepsilon(w)
\]
strongly in every subcritical local $L^p$ and weak-* against
$L^{\qwp,1}$. Thus $\mathcal T_\varepsilon$ is continuous on
$\mathcal C_H$. Since $\mathcal C_H$ is metrizable by
Lemma~\ref{lem:ball}, the sequential formulation suffices.

All hypotheses of the Schauder--Tychonoff fixed-point theorem are now
verified. The set $\mathcal C_H$ is nonempty, convex, and compact, and
$\mathcal T_\varepsilon$ is a continuous self-map of this set.
The Schauder--Tychonoff theorem \cite{Tychonoff1935} therefore gives the existence of a fixed point
\begin{equation}\label{eq:regfixed}
 \Omega_\varepsilon
 =\mathcal T_\varepsilon(\Omega_\varepsilon).
\end{equation}

It remains to remove the regularization and recover the limiting
profile together with the convergence properties asserted in the
theorem. By Lemma \ref{lem:ball}, we can
choose\phantomsection\label{def:profilesequence}
$\varepsilon_n\downarrow0$ so that
\(
 \Omega_n=\Omega_{\varepsilon_n}
 \rightharpoonup^*\Omega
\)
in $\mathcal C_H$. Then
\begin{equation}\label{eq:profilelimit}
 U_n=\eta_{\varepsilon_n}*(K*\Omega_n)\longrightarrow
 U=K*\Omega\quad\hbox{strongly in }L^1_{\loc}.
\end{equation}
For the varying mollifier, use the decomposition
\[
 U_n-U
 =
 \eta_{\varepsilon_n}*(K*\Omega_n-U)
 +
 \bigl(\eta_{\varepsilon_n}*U-U\bigr).
\]
On every fixed ball, the first term tends to zero by the local strong
convergence supplied by Lemma~\ref{lem:velocity}, after enlarging the
ball by the support radius of the mollifier. The second term tends to
zero by local approximation of $U$ by convolution. The critical bounds
remain uniform along the sequence.

The endpoint stability Lemma~\ref{lem:endpointstability} established above now gives local convergence
in measure of $L_{U_n}$ to $L_U$ and of their inverses. Since
\(
 \Omega_n=H\circ L_{U_n},
\)
Lemma~\ref{lem:composition} gives
\(
 \Omega_n\longrightarrow H\circ L_U
\)
strongly in every local $L^p$ with $1\leq p<\qw$ and weak-* against
$L^{\qwp,1}$. The weak-* limit is $\Omega$, which yields the limiting fixed-point
identity
\begin{equation}\label{eq:limfixed}
 \Omega=H\circ L_U.
\end{equation}
Measure preservation also shows that every $\Omega_n$ and $\Omega$
is equimeasurable with $H$. The resulting exact critical norms and
the velocity estimates are
\begin{equation}\label{eq:selectedbounds}
 \begin{gathered}
 \Omega_n=H\circ L_{U_n},\qquad
 \Mnorm{\Omega_n}{\qw}=\Mnorm{\Omega}{\qw}=M_H,\\
 \Mnorm{U_n}{\qu}+\Mnorm{DU_n}{\qw}
 +\Mnorm{U}{\qu}+\Mnorm{DU}{\qw}
 \leq C_\alpha M_H.
 \end{gathered}
\end{equation}
The convergence $\Omega_n\to\Omega$ is strong in every local
$L^p$, $1\leq p<\qw$, and weak-* against $L^{\qwp,1}$; $U_n\to U$
in local $L^1$. Both $L_{U_n}\to L_U$ and
$L_{U_n}^{-1}\to L_U^{-1}$ hold locally in measure.
These are precisely the identities, bounds, and convergence properties
asserted in Theorem~\ref{thm:lagrangianfixedpoint}. We retain this
selected sequence for the Euler realization in
Section~\ref{sec:realization}.
\end{proof}

\section{Euler realization and properties}\label{sec:realization}
Fix the limiting profile and the selected regularized sequence from
Section~\ref{sec:formulation}. The purpose of this section is to turn the
Lagrangian fixed-point construction into the Euler solution asserted in
Theorem~\ref{thm:main}. This requires more than the limiting profile
identity. The identity $\Omega=H\circ L_U$ determines the transported
vorticity, but by itself it does not justify the Euler velocity equation
in the critical class. The selected regularized sequence is used to
bridge this gap.

The argument has three stages. First, scaling covariance converts the
time-one endpoint identity into physical transport from time zero and
gives the corresponding regularized approximations and their strong
local convergence.
Second, we establish the Euler velocity equation. We symmetrize the
legitimate regularized flux before passing to the limit, obtaining a
bilinear vorticity identity without assuming local integrability of
the ordinary product $u\omega$. A frozen stress identity then recovers
the divergence-free velocity formulation, while the canonical pressure
and a decay argument eliminate the remaining harmonic residual.
Third, closed-time transport yields the strong vorticity and velocity
traces and supplies the initial term in the velocity equation. It also
gives renormalization and the stated conservation properties. We
conclude with the profile equation, optional symmetries, and examples
illustrating the angular hypothesis.

\subsection{From the profile to physical transport}\label{subsec:transport}
The profile identity concerns time one, whereas the theorem requires
transport from time zero on every finite time interval.
The self-similar scaling of the velocity is compatible with its
Lagrangian flow. More precisely, rescaling time by $c^\alpha$ and
space by $c$ transforms a trajectory into another trajectory of
the same velocity field. We refer to this property as the
scaling covariance of the Lagrangian flow.
The following lemma makes this property precise.
\begin{lemma}[Covariance through time zero]
For\phantomsection\label{lem:covariance} a divergence-free profile velocity $U\in\mathcal M^{\qu}$ with $DU\in\mathcal M^{\qw}$, the flow of $b_U$
satisfies, for every $c>0$ and fixed $s,t\geq0$,
\begin{equation}\label{eq:flowcovariance}
 X_{b_U}(c^\alpha t,c^\alpha s,cx)
       =cX_{b_U}(t,s,x)
 \quad\hbox{for almost every }x.
\end{equation}
In particular, for $t>0$,
\begin{equation}\label{eq:endpointcovariance}
 X_{b_U}(0,t,x)=t^{1/\alpha}L_U(t^{-1/\alpha}x).
\end{equation}
If $\Omega=H\circ L_U$, its physical scaling in
\eqref{eq:selfsimilar} equals $H\circ X_{b_U}(0,t,\cdot)$.
\end{lemma}
\begin{proof}
The drift obeys $b_U(c^\alpha t,cx)=c^{1-\alpha}b_U(t,x)$.
Rescaling a trajectory in both time and space therefore gives the
same integral ODE with the same initial label.  Conjugating a
measure-preserving map by a dilation preserves measure. Both flows
are absolutely continuous through zero, so prescribed-drift uniqueness
in Proposition~\ref{prop:flow} proves \eqref{eq:flowcovariance}, also
when either endpoint is zero. Taking the base interval from $1$ to $0$
and $c=t^{1/\alpha}$ gives \eqref{eq:endpointcovariance}. Finally,
$H(cy)=c^{-\alpha}H(y)$ and the fixed-point identity give
\begin{equation}\label{eq:physicalidentity}
 H(X_{b_U}(0,t,x))
 =t^{-1}H(L_U(t^{-1/\alpha}x))
 =t^{-1}\Omega(t^{-1/\alpha}x).
\end{equation}
Flows on larger finite intervals restrict to the same flow by
uniqueness. Thus one profile, rather than successive choices of
solutions, supplies every finite time horizon.
\end{proof}

We now apply this covariance to the limiting profile and to the entire
regularized sequence. This transfers the profile-level construction to
physical variables while preserving the approximation selected in
Subsection~\ref{subsec:fixed}.
Define $\omega,u$ by \eqref{eq:selfsimilar}, with
$\omega(0)=H$ and $u(0)=K*H$. Apply the lemma to $U,\Omega$ and to
each retained $U_n,\Omega_n$. For the approximations define
\begin{equation}\label{eq:physicalregularization}
 \begin{split}
 \omega_n(t,x)&=t^{-1}\Omega_n(t^{-1/\alpha}x),\qquad
 v_n(t)=K*\omega_n(t),\\
 b_n(t,x)&=t^{1/\alpha-1}U_n(t^{-1/\alpha}x)
       =\eta_{\delta_n(t)}*v_n(t,x),\qquad
 \delta_n(t)=\varepsilon_n t^{1/\alpha}.
 \end{split}
\end{equation}
Here $v_n$ is the canonical Biot--Savart velocity
associated with $\omega_n$, whereas $\omega_n$
is transported by the mollified velocity $b_n$,
not by $v_n$. The approximate pairs are not
asserted to solve the Euler equation.
The uniform profile bounds \eqref{eq:selectedbounds}
and the scaling identities \eqref{eq:criticalscaling}
give the common flow bounds \eqref{eq:flowbounds}
for the fields $b_n$. The convergence
\eqref{eq:profilelimit}, together with the time-splitting
argument in the proof of Lemma~\ref{lem:endpointstability},
gives local spacetime $L^1$ convergence of $b_n$ to $u$
on every $[0,T]$. The backward-flow stability in
Proposition~\ref{prop:flow}, followed by
Lemma~\ref{lem:composition}, then gives
\begin{equation}\label{eq:scalarlimit}
 \omega_n(t)=H\circ X_{b_n}(0,t,\cdot)\longrightarrow
 H\circ X_u(0,t,\cdot)=\omega(t)
 \quad\hbox{in }L^p_{\loc},\quad 1\leq p<\qw.
\end{equation}
The critical profile bounds \eqref{eq:selectedbounds},
the physical scaling \eqref{eq:physicalregularization},
and the exponent relation \eqref{eq:exponents}
give uniform $\mathcal M^{\qw}$ bounds for
$\omega_n(t)$ and $\omega(t)$. Applying
\eqref{eq:uniformintegrability} on each bounded ball
gives uniform local $L^p$ bounds for $1\leq p<\qw$.
Dominated convergence in time therefore upgrades
\eqref{eq:scalarlimit} to strong local spacetime
$L^p$ convergence.
The last equality uses \eqref{eq:limfixed} and
Lemma~\ref{lem:covariance}. Scaling the absolutely convergent
Biot--Savart integral gives $u(t)=K*\omega(t)$.

Thus \eqref{eq:exacttransport} is now proved, with the closed-time
invertible flow supplied by Proposition~\ref{prop:flow}. Measure
preservation and \eqref{eq:operators} give the critical bounds
\eqref{eq:mainbounds}. It remains to prove the velocity equation. The convergence
in \eqref{eq:scalarlimit} will be its scalar input.

\subsection{The Euler identity}\label{subsec:euler}
We now use the retained regularized sequence to pass from Lagrangian
transport to the Euler velocity equation. For each fixed
regularization, the transported scalar has a legitimate vorticity flux.
We first symmetrize this flux and pass to a limiting bilinear identity.
We then identify the bilinear form with the velocity stress and finally
recover the canonical pressure.

Applying the Lorentz--H\"older inequality
\eqref{eq:lorentzholder} with $\ell=\qu$,
together with the critical bounds
\eqref{eq:operators} and \eqref{eq:selectedbounds},
gives
\[
 \|b_n(t)\|_\infty
 \leq C_\alpha M_H
 \|\eta_{\delta_n(t)}\|_{L^{\qu',1}}
 \leq C_{\alpha,\eta}M_H\delta_n(t)^{-2/\qu}.
\]
Using the smoothing radius in
\eqref{eq:physicalregularization} and the exponent relation
\eqref{eq:exponents}, we obtain
\begin{equation}\label{eq:boundedregularized}
 \|b_n(t)\|_\infty
 \leq C_{\alpha,\eta}M_H\varepsilon_n^{1-\alpha}
                       t^{1/\alpha-1}.
\end{equation}
Since $1<\alpha<2$, the time singularity in
\eqref{eq:boundedregularized} is integrable at zero.
For each fixed $n$, this estimate, together with the
uniform local $L^1$ vorticity control obtained from
\eqref{eq:uniformintegrability},
\eqref{eq:selectedbounds}, and
\eqref{eq:physicalregularization}, gives
$b_n\omega_n\in L^1((0,T)\times B_R)$
for every finite $T,R$.
Integrating a compact smooth test along the absolutely continuous
flow and changing variables by its measure-preserving maps yields
\begin{equation}\label{eq:approxtransport}
 \int_0^\infty\!\int_{\R^2}
 \omega_n(\partial_t\zeta+b_n\cdot\nabla\zeta)\idiff x\idiff t
 +\int_{\R^2}H\zeta(0)\idiff x=0.
\end{equation}
Fubini is justified by the displayed local product estimate. The time
and initial terms are locally integrable as well. No time derivative
of the smoothing radius is used in this transport identity.

To pass this legitimate approximate identity, we use cancellation
before removing the regularization. The radial choice of mollifier
retains the geometric structure of the Biot--Savart kernel.
Symmetrization introduces a difference of test-function gradients,
which cancels the kernel singularity near the diagonal. The resulting
form will also admit a uniform tail bound, allowing the regularization
to be removed.
Write\phantomsection\label{def:smoothedkernel}
 $K_\delta=\eta_\delta*K$ for a positive spatial radius $\delta$.
Radial symmetry and circulation around a circle give
\begin{equation}\label{eq:regularizedkernel}
 K_\delta(z)=K(z)\int_{|y|<|z|}\eta_\delta(y)\idiff y,
 \qquad |z|\,|K_\delta(z)|\leq\frac1{2\pi}.
\end{equation}
To be specific, $K_\delta$ is a rotationally equivariant tangential
field with curl $\eta_\delta$. Stokes' formula makes its circulation
the mass inside the circle, proving the identity. Smoothness rules
out an additional point circulation. In particular it is odd and
converges to $K$ away from zero.

We use the standard symmetrized weak-vorticity form
\cite{Delort1991,Schochet1995}, which avoids requiring local
integrability of the ordinary vorticity flux. For a compact
smooth scalar spatial test $\zeta$, define
\begin{equation}\label{eq:bilinear}
 \mathcal B_w(\zeta)=\frac12\iint_{\R^2\times\R^2}
 w(x)w(y)K(x-y)\cdot
       \bigl(\nabla\zeta(x)-\nabla\zeta(y)\bigr)\idiff x\idiff y.
\end{equation}
This is absolutely convergent for $w\in\mathcal M^{\qw}$.
On a bounded product set the kernel is bounded by
$\|D^2\zeta\|_\infty/(2\pi)$. Outside $B_L\times B_L$, with
$\supp\nabla\zeta\subset B_R$ and $L>2R$, at least one point must
belong to $B_R$ for a nonzero contribution, and the other contributes at most $C_\zeta/|y|$.
Thus \eqref{eq:tail} bounds the contribution from pairs with at least one point outside $B_L$ by
\begin{align}\label{eq:bilineartail}
\begin{split}
&\frac12
\iint_{(\mathbb R^2\times\mathbb R^2)\setminus(B_L\times B_L)}
|w(x)w(y)|\,
\left|
K(x-y)\cdot\bigl(\nabla\zeta(x)-\nabla\zeta(y)\bigr)
\right|\,dx\,dy
\\
&\quad\leq
C\|\nabla\zeta\|_\infty
\left(\int_{B_R}|w(x)|\,dx\right)
\left(\int_{|y|>L}\frac{|w(y)|}{|y|}\,dy\right)
\\
&\quad\leq
C_\alpha\|\nabla\zeta\|_\infty
|B_R|^{1-1/q_\omega}
\|w\|_{\mathcal M^{q_\omega}}^2
L^{1-\alpha}.
\end{split}
\end{align}
Both bounds hold uniformly with $K_\delta$ in place of $K$ by
\eqref{eq:regularizedkernel}.

For fixed $n$ and $t>0$, the unsymmetrized spatial
double integral in the flux of \eqref{eq:approxtransport}
is absolutely integrable. Near the diagonal, the
regularized kernel is bounded by its representation
\eqref{eq:regularizedkernel}. At infinity, the tail
estimate \eqref{eq:tail} gives the required integrability.
Exchanging the spatial variables and using the oddness
of $K_{\delta_n(t)}$ then gives
\begin{equation}\label{eq:symapprox}
 \int\omega_n b_n\cdot\nabla\zeta\idiff x
 =\frac12\iint\omega_n(t,x)\omega_n(t,y)K_{\delta_n(t)}(x-y)
 \cdot\bigl(\nabla\zeta(x)-\nabla\zeta(y)\bigr)\idiff x\idiff y.
\end{equation}
Thus \eqref{eq:symapprox} identifies the regularized
vorticity flux with the mollified version of the
bilinear form \eqref{eq:bilinear}.
These spatial identities can then be integrated in time because
\eqref{eq:boundedregularized} already justifies the original flux.

For $\delta>0$, define the mollified bilinear form by
\[
 \mathcal B_w^\delta(\zeta)
 =\frac12\iint_{\R^2\times\R^2}
 w(x)w(y)K_\delta(x-y)\cdot
 \bigl(\nabla\zeta(x)-\nabla\zeta(y)\bigr)
 \idiff x\idiff y.
\]
The identity \eqref{eq:symapprox} identifies the
regularized vorticity flux with this bilinear form.
The following lemma provides the continuity needed
to remove the regularization.

\begin{lemma}[Continuity of the symmetrized vorticity form]
\label{lem:symmetrizedcontinuity}
Suppose that $w_n$ are real-valued functions uniformly
bounded in $\mathcal M^{\qw}(\R^2)$, with
$w_n\to w$ strongly in $L^1_{\loc}(\R^2)$.
If $\delta_n>0$ and $\delta_n\to0$, then, for every
$\zeta\in C_c^\infty(\R^2)$,
\[
 \mathcal B_{w_n}^{\delta_n}(\zeta)
 \longrightarrow \mathcal B_w(\zeta).
\]
The same conclusion holds in spacetime. More precisely,
suppose that $w_n$ are time-dependent functions satisfying
\[
 \sup_n\esssup_{0<t<T}\Mnorm{w_n(t)}{\qw}<\infty
\]
and converging strongly to $w$ in
$L^1((0,T)\times B_R)$ for every $R>0$.
If $\delta_n(t)>0$ converges to zero for almost every
$t\in(0,T)$, then, for every
$\zeta\in C_c^\infty([0,T)\times\R^2)$,
\[
 \int_0^T
 \mathcal B_{w_n(t)}^{\delta_n(t)}(\zeta(t))
 \idiff t
 \longrightarrow
 \int_0^T\mathcal B_{w(t)}(\zeta(t))\idiff t.
\]
\end{lemma}

\begin{proof}
The uniform critical bounds pass to the limiting
function by Fatou's lemma. Fix a spatial test $\zeta$
and choose $R>0$ such that
$\supp\nabla\zeta\subset B_R$, followed by $L>2R$.
By \eqref{eq:regularizedkernel}, the difference
of test-function gradients cancels the kernel
singularity and gives the uniform bound
\[
 \left|
 K_\delta(x-y)\cdot
 \bigl(\nabla\zeta(x)-\nabla\zeta(y)\bigr)
 \right|
 \leq\frac{\|D^2\zeta\|_\infty}{2\pi}.
\]
The same bound holds with $K$ in place of $K_\delta$.
On $B_L\times B_L$, the difference of scalar
products satisfies
\[
\begin{aligned}
&\iint_{B_L\times B_L}
 |w_n(x)w_n(y)-w(x)w(y)|
 \idiff x\idiff y\\
&\quad\leq
 \bigl(\|w_n\|_{L^1(B_L)}
       +\|w\|_{L^1(B_L)}\bigr)
 \|w_n-w\|_{L^1(B_L)}
 \longrightarrow0.
\end{aligned}
\]
For the fixed limiting product, the convergence
of $K_{\delta_n}$ to $K$ away from the diagonal
and the uniform kernel bound allow passage to
the limit by dominated convergence.
The contribution from outside $B_L\times B_L$
is uniformly bounded by \eqref{eq:bilineartail}
and tends to zero as $L\to\infty$.
This proves the spatial assertion.

For the spacetime assertion, the same product
estimate, integrated in time, gives
\begin{equation}\label{eq:productlimit}
\begin{split}
&\int_0^T\!\iint_{B_L\times B_L}
 |w_n(t,x)w_n(t,y)-w(t,x)w(t,y)|
 \idiff x\idiff y\idiff t\\
&\quad\leq
 \int_0^T
 \bigl(\|w_n(t)\|_{L^1(B_L)}
       +\|w(t)\|_{L^1(B_L)}\bigr)
 \|w_n(t)-w(t)\|_{L^1(B_L)}
 \idiff t
 \longrightarrow0.
\end{split}
\end{equation}
Here the local $L^1$ norms are uniformly bounded
by \eqref{eq:uniformintegrability} and the
common critical bounds.
For the fixed limiting product, the convergence
of $K_{\delta_n(t)}$ to $K$ for almost every
positive time allows dominated convergence
on $(0,T)\times B_L\times B_L$.
The uniform tail estimate \eqref{eq:bilineartail},
integrated over the finite time interval,
controls the remaining contribution by
$C L^{1-\alpha}$.
Letting $L\to\infty$ completes the proof.
\end{proof}

The uniform critical bounds \eqref{eq:selectedbounds},
the scaling identities \eqref{eq:physicalregularization},
and the strong local spacetime convergence following
\eqref{eq:scalarlimit} verify the spacetime hypotheses
of Lemma~\ref{lem:symmetrizedcontinuity}.
Moreover, the smoothing radii in
\eqref{eq:physicalregularization} tend to zero
at every positive time.
Applying Lemma~\ref{lem:symmetrizedcontinuity}
to \eqref{eq:symapprox} and passing to the limit
in \eqref{eq:approxtransport} gives
\begin{equation}\label{eq:weakvorticity}
 \int_0^\infty\!\int_{\R^2}
 \omega\partial_t\zeta\idiff x\idiff t
 +\int_0^\infty
 \mathcal B_{\omega(t)}(\zeta(t))\idiff t
 +\int_{\R^2}H\zeta(0)\idiff x=0.
\end{equation}
This identity is justified without forming
the ordinary limiting vorticity flux.

The limit is a vorticity identity in symmetrized form. To convert it
to velocity we need a velocity identity valid for each frozen scalar
input. Its proof uses auxiliary compact truncations, independently
of the physical transport sequence.
\begin{lemma}[Frozen stress identity]
For\phantomsection\label{lem:stress} every real $w\in\mathcal M^{\qw}$ and compact smooth $\zeta$, with
$v=K*w$, one has
\begin{equation}\label{eq:stressidentity}
 \mathcal B_w(\zeta)=-\int_{\R^2}(v\otimes v):D(\nabla^\perp\zeta)\idiff x.
\end{equation}
No parity assumption on $w$ is required.
\end{lemma}
\begin{proof}
Choose\phantomsection\label{def:frozentruncations}
 a smooth cutoff $0\leq\chi\leq1$, equal to one on $B_1$ and
supported in $B_2$, and put
$\widehat w_n=\eta_{1/n}*(\chi(\cdot/n)w)$. These functions are smooth and
compactly supported, have uniformly bounded $\mathcal M^{\qw}$ norms,
and converge strongly locally in every $L^p$, $p<\qw$. Predual density
also gives weak-* convergence, so Lemma~\ref{lem:velocity} gives
$K*\widehat w_n\to v$ strongly in local $L^2$. For a smooth divergence-free
velocity with curl $\widehat w_n$, the identity
\begin{equation}\label{eq:smoothcurl}
 \curl\operatorname{div}(\widehat v_n\otimes\widehat v_n)
       =\operatorname{div}(\widehat v_n\widehat w_n),\qquad
       \widehat v_n=K*\widehat w_n,
\end{equation}
follows by differentiating components and using $\operatorname{div}\widehat v_n=0$.
Our convention gives
$\langle\curl\widehat v_n,\zeta\rangle=-\int\widehat v_n\cdot\nabla^\perp\zeta$.
Testing \eqref{eq:smoothcurl} and symmetrizing
its right side gives \eqref{eq:stressidentity}
for $\widehat w_n$ with the displayed minus sign.
The bounded product-set argument in the proof of
Lemma~\ref{lem:symmetrizedcontinuity}, together
with the uniform tail estimate
\eqref{eq:bilineartail}, passes the bilinear
form to the limit. The velocity compactness
\eqref{eq:velocitycompact} gives strong local
$L^2$ convergence and therefore permits passage
to the limit in the stress term.
These auxiliary truncations are used only for
the frozen identity, not as Euler approximations.
\end{proof}


Substituting Lemma~\ref{lem:stress} into \eqref{eq:weakvorticity}
and using $\curl u=\omega$ proves the velocity identity for tests
$\nabla^\perp\zeta$, with initial velocity $K*H$. Every smooth compactly
supported divergence-free vector field in the plane has such a
compactly supported streamfunction. Indeed the associated closed
one-form has a global potential. Outside a ball containing the support
that potential is constant, and subtracting that constant makes it
compactly supported. Thus the divergence-free test class is complete.

The identity for divergence-free tests shows that
\[
\partial_tu+\operatorname{div}(u\otimes u)
\]
is a spatial gradient. To recover the full velocity equation, we choose the
canonical pressure in \eqref{eq:pressure} and show that
the remaining residual vanishes.
Since $\qu>2$, the critical velocity bound
\eqref{eq:mainbounds} and the weak-Lorentz product
estimate give a uniform $\mathcal M^{\qu/2}$ bound
for $u\otimes u$. The boundedness of Riesz transforms
on this space then gives the pressure estimate
in \eqref{eq:pressure} and the Poisson identity
\begin{equation}\label{eq:pressurepoisson}
 \Delta P=-\sum_{i,j=1}^2\partial_i\partial_j(u_i u_j).
\end{equation}
For a compact smooth time test\phantomsection\label{def:timetests}
$\vartheta\in C_c^\infty((0,\infty))$, define the residual and its
time average by
\begin{equation}\label{eq:residual}
 \mathcal A=\partial_tu+\operatorname{div}(u\otimes u)+\nabla P,
 \qquad \mathcal A_\vartheta=\langle\mathcal A,\vartheta\rangle_t.
\end{equation}
The first has both divergence and curl zero on $(0,\infty)\times\R^2$.
Thus $\mathcal A_\vartheta$ is an entire harmonic vector
distribution, hence a smooth harmonic vector field.

Choose\phantomsection\label{def:averagingkernel}
 a radial smooth unit-mass kernel $\kappa$ supported in $B_1$
and set $\kappa_R(x)=R^{-2}\kappa(x/R)$. The mean-value property
gives $\mathcal A_\vartheta=\mathcal A_\vartheta*\kappa_R$.
Moving derivatives to the kernel and using Lorentz H\"older \eqref{eq:lorentzholder} yields,
at every fixed spatial point,
\begin{equation}\label{eq:harmonicdecay}
 |\mathcal A_\vartheta(x)|
 \leq C_{\alpha,\kappa}
 \left(\|\vartheta'\|_1M_H R^{-2/\qu}
       +\|\vartheta\|_1M_H^2R^{-1-4/\qu}\right).
\end{equation}
The powers follow from
$\|\kappa_R\|_{L^{\qu',1}}
=R^{-2/\qu}\|\kappa\|_{L^{\qu',1}}$ and
$\|\nabla\kappa_R\|_{L^{(\qu/2)',1}}
=R^{-1-4/\qu}\|\nabla\kappa\|_{L^{(\qu/2)',1}}$.
Both terms vanish as $R\to\infty$, so $\mathcal A_\vartheta=0$.
Since $\vartheta$ was arbitrary, $\mathcal A=0$.
This argument eliminates the entire harmonic residual, including
spatial constants, without parity. All curl and divergence identities
hold on the whole plane, so there is no omitted point force.
The initial formulation follows from the traces proved next.

\subsection{Traces, conservation, and further properties}\label{subsec:class}
The interior velocity equation has been established independently of
its initial trace. We now use the closed-time transport to prove the
two strong continuity assertions, then recover the initial term and
complete renormalization and conservation.

Fix $s\geq0$. The two-time flow law from
Proposition~\ref{prop:flow} and 
the transport identity
\eqref{eq:exacttransport} give
\begin{equation}\label{eq:shorttransport}
 \omega(t)=\omega(s)\circ X_u(s,t,\cdot).
\end{equation}
The displacement estimate \eqref{eq:displacement}
shows that $X_u(s,t,\cdot)$ converges locally in
measure to the identity as $t\to s$, in either
time direction. The escape estimate \eqref{eq:escape}
also verifies the tightness condition
\eqref{eq:maptightness}. Lemma~\ref{lem:composition},
applied with $f=\omega(s)$, now gives strong local
$L^p$ continuity for every $1\leq p<\qw$,
including at $s=0$.
The global weak norm stays equal to $M_H$ by exact transport.

For velocity continuity, fix an observation ball $B_R$ and choose
$1<\ell<\qw$. On $B_L$, the scalar difference
$\omega(t)-\omega(s)$ tends to zero in $L^\ell$. Note that
\[
\left\|
K*\bigl[(\omega(t)-\omega(s))\mathbf1_{B_L}\bigr]
\right\|_{L^{2\ell/(2-\ell)}(\mathbb R^2)}
\leq
C_\ell
\|\omega(t)-\omega(s)\|_{L^\ell(B_L)}
\longrightarrow0.
\]
Since $2\ell/(2-\ell)>2$, it therefore tends to zero in $L^2(B_R)$.
For $L>2R$, the remaining velocity tail is uniformly bounded on $B_R$
by $C_\alpha M_H L^{1-\alpha}$, using \eqref{eq:tail}.
Taking $t\to s$ first and then $L\to\infty$
proves the velocity continuity assertion in
\eqref{eq:traces}, including the initial trace
$u(0)=K*H$. Since $\qu>2$, the bounds
\eqref{eq:mainbounds} and \eqref{eq:pressure},
together with \eqref{eq:uniformintegrability},
give uniform local $L^1$ bounds for the stress
and pressure up to time zero. A temporal cutoff
in the interior velocity identity, combined with
the strong velocity trace, then yields
\eqref{eq:velocity} for arbitrary tests meeting $t=0$.

For bounded $\beta$ as in the theorem, the exact transport identity  gives
$\beta(\omega(t))=\beta(H)\circ X_u(0,t,\cdot)$.
Since $\beta$ is bounded and $u$ is locally
integrable, the renormalized flux
$u\beta(\omega)$ belongs to $L^1_{\loc}$,
even when the ordinary flux $u\omega$ does not.
Differentiate a compact test along the absolutely continuous particle
paths and integrate over initial labels. The resulting integrals are
absolutely convergent because $\beta$ is bounded, $u$ is locally
spacetime $L^1$, and the flow preserves measure. In particular,
the velocity term is bounded by
$\|\beta\|_\infty\int|u\cdot\nabla\zeta|\idiff x\idiff t$.
Fubini therefore gives \eqref{eq:renormalization}, including its
initial term. Its local strong trace also follows from the Lipschitz
continuity of $\beta$ and \eqref{eq:traces}.
Finally, invertible measure preservation proves exact signed
equimeasurability and change of variables for every absolutely
integrable $\mathcal F(H)$. Lemma~\ref{lem:distribution} supplies the constants
in \eqref{eq:signeddistribution}. This proves Theorem~\ref{thm:main}.

With Theorem~\ref{thm:main} proved, we record several structural
consequences of the same construction. These concern the profile
velocity equation, optional symmetries preserved by the fixed-point
scheme, and examples illustrating the scope of the angular datum class.

The\phantomsection\label{def:profilepressure}
pressure is self-similar as well. Riesz transforms commute with
dilation, so if
$\Pi=\sum_{i,j}\Riesz_i\Riesz_j(U_iU_j)$, then
\begin{equation}\label{eq:profilesystem}
 \begin{gathered}
 P(t,x)=t^{2/\alpha-2}\Pi(t^{-1/\alpha}x),\\
 (1/\alpha-1)U-\alpha^{-1}(y\cdot\nabla)U
       +\operatorname{div}(U\otimes U)+\nabla\Pi=0,
 \qquad \operatorname{div}U=0.
 \end{gathered}
\end{equation}
The profile equation is distributional. It follows by the change of
variables in the velocity identity, rather than by an assumed
untruncated profile-vorticity equation.

The fixed-point construction is compatible with orthogonal symmetries
that preserve the datum with the appropriate pseudoscalar covariance.

\begin{proposition}[Optional symmetries]
Let\phantomsection\label{prop:symmetry} $\mathcal G$ be a subgroup of the orthogonal group $\mathrm O(2)$ such that
$H(Ox)=\det(O)H(x)$ almost everywhere for every $O\in\mathcal G$.
The profile may be chosen with
\begin{equation}\label{eq:symmetry}
 \Omega(Oy)=\det(O)\Omega(y),\qquad U(Oy)=OU(y),
 \qquad O\in\mathcal G.
\end{equation}
The corresponding physical fields and flow have the same covariance,
with almost-everywhere identities for each fixed group element and time.
\end{proposition}
\begin{proof}
The conditions $w(Ox)=\det(O)w(x)$ are weak-* closed linear conditions.
Their intersection with $\mathcal C_H$ is compact, convex, and contains
zero. Since $(Ox)^\perp=\det(O)O x^\perp$, one has
$K(Oz)=\det(O)OK(z)$. A change of variables then gives
$(K*w)(Ox)=O(K*w)(x)$ for such a $w$. Radial mollification preserves
this vector covariance, and uniqueness of the prescribed-drift flow
gives $L_U(Ox)=OL_U(x)$. Hence $H\circ L_U$ has the required signed
covariance, and the entire fixed-point and limiting construction stays
in this invariant domain. The same uniqueness gives physical flow
covariance. No nonlinear solutions have been averaged.
\end{proof}

The proof applies, in particular, to $h=1$, $h=\cos\theta$,
$h=1+\cos\theta$, bounded angular jumps, and all trigonometric
polynomials. It also applies to an angular cusp equal to
$|\theta|^{-\alpha/4}$ near zero and bounded elsewhere, since its
$\qw$th power is locally $|\theta|^{-1/2}$. No step differentiates $h$.
By contrast, a cusp $|\theta|^{-\alpha/2}$ has divergent critical
angular integral and does not belong to the planar weak class.

The constant angular datum also tests the nonlinear formulation.
The stationary solution in Appendix~\ref{app:comparison},
\eqref{eq:radialexample}, has $u\omega\notin L^1_{\loc}$ whenever
$\alpha\geq3/2$, including the endpoint, while its velocity equation
is unforced on the whole plane. The need for symmetrization therefore
occurs within the theorem's datum class.

\par\bigskip
\noindent\textbf{Acknowledgments.}
This research was supported in part by the National Science Foundation under Grant No. DMS-2508463. The authors also gratefully acknowledge the generous support of the Choi Family Gift Fund and the Mike Yan Gift Fund.
The authors are grateful to Hyungjun Choi for his helpful comments
on the manuscript.
His insights into the connection to regularized Euler equations
and his explanations of the renormalization and Lagrangian properties
of the solutions constructed in \cite{Choi2026}
have been particularly illuminating.

\appendix
\section{Construction and stability of the closed-time flow}\label{app:flow}
This appendix establishes Proposition~\ref{prop:flow} directly in the
global weak-Lorentz setting. The argument uses the Sobolev pointwise
maximal-function inequality and the logarithmic-distance method of
\cite[Lemma~A.3, Theorem~3.8, and Corollaries~3.7 and~3.9]{CrippaDeLellis2008}.

Assume first that $b$ is smooth, bounded, divergence-free, and
satisfies \eqref{eq:flowbounds}. Its classical flow is a
measure-preserving diffeomorphism. Integrating the speed along
particles starting in a set $E$ of finite measure gives
\begin{equation}\label{eq:flowintegral}
 \begin{split}
 \int_E\sup_{t'\in I_{s,t}}|X_b(t',s,x)-x|\idiff x
 \leq\int_{I_{s,t}}\int_{X_b(t',s,E)}|b(t',y)|\idiff y\idiff t'
 \leq C_v|t-s|\,|E|^{1-1/\qu}.
 \end{split}
\end{equation}
Taking the supremum over $E$ yields the displacement bound
\eqref{eq:displacement} in the smooth case. A trajectory starting
in $B_R$ and leaving $B_L$ must travel at least $L-R$.
Chebyshev's inequality therefore gives the smooth escape bound
\begin{equation}\label{eq:escape0}
 |\{x\in B_R:
 \sup_{t'\in I_{s,t}}|X_b(t',s,x)|>L\}|
 \leq\left(\frac{C_v|t-s|}{L-R}\right)^{\qu}.
\end{equation}
Both bounds apply in either time direction. For smooth flows,
invertibility and measure preservation turn the backward escape
bound into the following estimate for particles arriving from infinity,
\begin{equation}\label{eq:inverseescape}
 \int_{|x|>L}\ind_{B_R}(X_b(t,s,x))\idiff x
 =|\{y\in B_R:|X_b(s,t,y)|>L\}|
 \leq\left(\frac{C_vT}{L-R}\right)^{\qu}.
\end{equation}

Extend $b$ by zero outside $[0,T]$ and mollify in time and
space with nonnegative unit-mass kernels. For each $n$, the
approximating field $b^{(n)}$ is smooth, globally bounded, and
has bounded spatial derivatives. It remains divergence-free and
satisfies the common bounds \eqref{eq:flowbounds}. Indeed,
spatial convolution is contractive by
\eqref{eq:mollificationcontraction}, and temporal convolution
preserves the essential supremum of the Marcinkiewicz norms.
Moreover, $b^{(n)}\to b$ in
$L^1_{\loc}([0,T]\times\R^2)$. Write $X^{(n)}=X_{b^{(n)}}$.
The\phantomsection\label{def:maximal}
Hardy--Littlewood maximal operator $\mathsf M$ gives
$G_n=\mathsf M|Db^{(n)}|$, with
$\Mnorm{G_n(t)}{\qw}\leq C_\alpha C_\nabla$.
The Sobolev pointwise inequality takes the form
\begin{equation}\label{eq:sobolevdifference}
 |b^{(n)}(t,x)-b^{(n)}(t,y)|
 \leq C_{\mathrm{Sob}}|x-y|(G_n(t,x)+G_n(t,y)),
\end{equation}
where $C_{\mathrm{Sob}}$ depends only on the dimension.

Fix\phantomsection\label{def:logparameters}
labels in $B_R$ and discard those for which either smooth trajectory
leaves $B_L$, where $L>R$. The measure of the discarded labels is
at most twice the bound in \eqref{eq:escape0}. On the remaining
labels, the chain rule and \eqref{eq:sobolevdifference} give,
for every $\sigma>0$ and almost every $t\in[0,T]$,
\[
\begin{aligned}
&\frac{\mathrm d}{\mathrm dt}
 \ln\left(1+
 \frac{|X^{(n)}(t,0,x)-X^{(m)}(t,0,x)|}{\sigma}
 \right)\\
&\leq C_{\mathrm{Sob}}
 \bigl(G_n(t,X^{(n)}(t,0,x))
       +G_n(t,X^{(m)}(t,0,x))\bigr)+
 \frac{|b^{(n)}-b^{(m)}|
       (t,X^{(m)}(t,0,x))}{\sigma}
\end{aligned}
\]
for almost every $t\in[0,T]$.
Since both flows preserve Lebesgue measure and the retained
trajectories remain in $B_L$, the maximal-function terms
are controlled by
\[
 \int_0^T\int_{B_L}G_n(t,y)\idiff y\idiff t
 \leq C_\alpha TC_\nabla|B_L|^{1-1/\qw}.
\]
The difference-field term is bounded by its
$L^1((0,T)\times B_L)$ norm divided by $\sigma$.
Both trajectories start at the same label $x$.
Integrating in time and labels, taking the supremum in time,
and applying Chebyshev's inequality therefore give
\begin{equation}\label{eq:logestimate}
 \begin{split}
 &|\{x\in B_R:
 \sup_{0\leq t\leq T}
 |X^{(n)}(t,0,x)-X^{(m)}(t,0,x)|>\gamma\}|\\
 &\quad\leq
 2\left(\frac{C_vT}{L-R}\right)^{\qu}
 +\frac{C_\alpha TC_\nabla|B_L|^{1-1/\qw}
       +\|b^{(n)}-b^{(m)}\|_{L^1((0,T)\times B_L)}/\sigma}
       {\ln(1+\gamma/\sigma)}
 \end{split}
\end{equation}
for every $\gamma>0$.
For each fixed $L$, choose
\(
 \sigma=
 \left(
 \|b^{(n)}-b^{(m)}\|_{L^1((0,T)\times B_L)}
 +1/{n^2}+1/{m^2}
 \right)^{1/2}.
\)
Since $b^{(n)}\to b$ in local spacetime $L^1$,
both $\sigma$ and
\(
 \sigma^{-1}\|b^{(n)}-b^{(m)}\|_{L^1((0,T)\times B_L)}
\)
tend to zero as $n,m\to\infty$. The second term in
\eqref{eq:logestimate} therefore vanishes for fixed $L$.
Letting $L\to\infty$ then removes the escape term.
Thus the flows are Cauchy locally in measure with values in
$C([0,T];\R^2)$. Along a diagonal subsequence, the trajectories
converge for almost every initial label, uniformly in time.
The argument applies on any interval with a fixed starting time,
and time reversal gives the corresponding backward limits.
Fatou's lemma passes the smooth displacement bound derived from
\eqref{eq:flowintegral} to the limiting trajectories. In particular,
\eqref{eq:displacement} and \eqref{eq:escape} hold for the limiting
forward and backward maps.

To establish measure preservation for the limiting maps, fix two times $s,t$ and a nonnegative
$\varphi\in C_c(\R^2)$ supported in $B_R$.
For the smooth approximating flows, measure preservation
and \eqref{eq:inverseescape} give, for $L>R$,
\[
\begin{aligned}
&\left|
\int_{\R^2}\varphi(y)\idiff y
-
\int_{B_L}
\varphi(X^{(n)}(t,s,x))\idiff x
\right|\leq
\|\varphi\|_\infty
\left(\frac{C_vT}{L-R}\right)^{\qu}.
\end{aligned}
\]
For fixed $L$, local convergence in measure
allows passage to the limit in the integral over $B_L$.
The same tail bound passes to the limiting map
by Fatou's lemma. Letting $L\to\infty$ then gives
\[
\int_{\R^2}
\varphi(X_b(t,s,x))\idiff x
=
\int_{\R^2}\varphi(y)\idiff y.
\]
Thus the limiting maps preserve Lebesgue measure.
The same argument applies to the backward limits.

The integral ODE follows by passing to the limit in the
composed velocities. Fix $R>0$ and $L>R$ and discard labels in
$B_R$ for which either the approximating or the limiting path
leaves $B_L$. Their measure is bounded by \(
 2\big((C_vT)/(L-R)\big)^{\qu}.
\)
On the remaining labels, the difference
between the composed velocities 
$|b^{(n)}(t,X^{(n)}(t,0,x))-b(t,X_b(t,0,x))|$ is bounded by
\[
\begin{aligned}
 |(b^{(n)}-b)(t,X^{(n)}(t,0,x))|+
 |b(t,X^{(n)}(t,0,x))
       -b(t,X_b(t,0,x))|.
\end{aligned}
\]
Measure preservation bounds the integral of the first term
by $\|b^{(n)}-b\|_{L^1((0,T)\times B_L)}$, which tends to zero.
For the second term, approximate $b$ in
$L^1((0,T)\times B_L)$ by a smooth bounded vector field.
Measure preservation controls the two errors after composition,
while the convergence of trajectories in measure handles the
smooth approximation. Sending its approximation error to zero
proves convergence of the second term on the retained labels.
The discarded contributions are controlled uniformly. For every
finite-measure set $E$ of initial labels, measure preservation
and \eqref{eq:marcinkiewicz} give
\begin{equation}\label{eq:discardedvelocity}
 \int_E\int_0^T|b^{(n)}(t,X^{(n)}(t,0,x))|\idiff t\idiff x
 \leq C_vT|E|^{1-1/\qu}.
\end{equation}
The same estimate holds for the limiting map because the
pushforward of the restricted label measure has density at most
one and total mass $|E|$. By the escape estimates, the measure
of the discarded labels tends to zero uniformly as $L\to\infty$.
Consequently, the composed velocities converge in
$L^1((0,T)\times B_R)$. Passing to the limit in the smooth
integral equations proves \eqref{eq:rlf} for the limiting
trajectories, including absolute continuity through both endpoints.

Almost-everywhere invertibility follows by passing the smooth
inverse identities to the limit. For every fixed $s,t$, the
smooth flows satisfy
 $X^{(n)}(s,t,X^{(n)}(t,s,x))=x$. 
For $L>R$ and $\gamma>0$, measure preservation gives
\[
\begin{aligned}
&\bigl|\{x\in B_R:
 |X^{(n)}(s,t,X^{(n)}(t,s,x))
 -X_b(s,t,X^{(n)}(t,s,x))|>\gamma\}\bigr|\\
&\quad\leq
 |\{x\in B_R:|X^{(n)}(t,s,x)|>L\}|+
 |\{y\in B_L:
 |X^{(n)}(s,t,y)-X_b(s,t,y)|>\gamma\}|.
\end{aligned}
\]
The first term is uniformly small for large $L$ by
\eqref{eq:escape0}. For fixed $L$, the second tends to zero
by local convergence in measure of the backward maps.
It remains to compare
$X_b(s,t,X^{(n)}(t,s,x))$ with
$X_b(s,t,X_b(t,s,x))$.
For the remaining comparison, approximate the measurable map
$X_b(s,t,\cdot)$ in measure on $B_L$ by a bounded uniformly
continuous map. Measure preservation controls the preimages of
the approximation-error sets under both
$X^{(n)}(t,s,\cdot)$ and $X_b(t,s,\cdot)$.
Local convergence in measure of the forward maps handles the
continuous approximation, and \eqref{eq:escape0} and
\eqref{eq:escape} control intermediate images outside $B_L$.
Taking $n\to\infty$, then reducing the approximation error,
and finally letting $L\to\infty$ proves convergence of the
composed maps in local measure. Passing to the limit yields the inverse identity
$X_b(s,t,X_b(t,s,x))=x$
for almost every $x$. Interchanging $s$ and $t$ gives the
other inverse identity.

Uniqueness holds among regular Lagrangian flows with bounded
compression. For the rough field $b$, the pointwise Sobolev
inequality underlying \eqref{eq:sobolevdifference} holds outside
spatial null sets at almost every time. Bounded compression and
Fubini's theorem justify its evaluation along almost every
trajectory. The logarithmic-distance argument leading to
\eqref{eq:logestimate} therefore applies to two flows of the
same field, with no difference-field term. For fixed $L$,
letting $\sigma\downarrow0$ removes the logarithmic term.
The required escape control follows from the rearrangement
bound \eqref{eq:marcinkiewicz} applied to the restricted
pushforward measures, whose densities are bounded by the
compression constants. Letting $L\to\infty$ then proves
uniqueness.

The construction applies to every fixed starting time.
Concatenation of trajectories and uniqueness give the
two-time flow law
$X_b(t,s)\circ X_b(s,r)=X_b(t,r)$
for each fixed choice of $r,s,t$ and almost every $x$.
A jointly measurable two-time representative follows from the
standard Sobolev regular Lagrangian flow theory under the local
Sobolev and growth conditions verified using
\eqref{eq:flowbounds} and \eqref{eq:growth}.
Uniqueness identifies this representative with the fixed-time
limits constructed above. The map identities hold for fixed
times modulo spatial null sets, as in Theorem~\ref{thm:main}.

Stability follows from the same logarithmic-distance argument.
For any fields $b^{(n)}\to b$ in local spacetime $L^1$ with
common bounds \eqref{eq:flowbounds}, the difference-field
contribution is bounded by
$\|b^{(n)}-b\|_{L^1((0,T)\times B_L)}/\sigma$.
As in the Cauchy argument, for fixed $L$ one can choose
$\sigma>0$ tending to zero while this quotient also tends to
zero. The logarithmic term then vanishes, and sending
$L\to\infty$ controls escape. This proves local convergence
in measure of the trajectories uniformly on closed time
intervals. The time-reversed argument gives the corresponding
backward stability, including both endpoints.

\section{Angular velocity, comparison, and a stationary example}\label{app:comparison}
The main construction uses canonical convolution without resolving
angular modes. Here an angular stream equation identifies the initial
velocity and permits a precise comparison with velocity-based existence
results. We finish with the stationary example used to test the
ordinary vorticity flux.

\subsection{Angular velocity and datum classes}
{}
Use the\phantomsection\label{def:polarbasis} polar basis
$e_\rho=(\cos\theta,\sin\theta)$ and
$e_\theta=e_\rho^\perp=(-\sin\theta,\cos\theta)$.

\begin{proposition}
For\phantomsection\label{prop:angularvelocity} $h\in L^{\qw}(\mathbb S^1)$, there is a unique periodic
$g\in W^{2,\qw}(\mathbb S^1)$ solving
\begin{equation}\label{eq:angularstream}
 g''+(2-\alpha)^2g=h.
\end{equation}
The canonical initial velocity has angular $W^{1,\qw}$ regularity and is
\begin{equation}\label{eq:initialvelocity}
 \begin{split}
 \Psi_H(\rho,\theta)&=\rho^{2-\alpha}g(\theta),\\
 K*H=\nabla^\perp\Psi_H
 &=\rho^{1-\alpha}
       \bigl(-g'(\theta)e_\rho+(2-\alpha)g(\theta)e_\theta\bigr).
 \end{split}
\end{equation}
Every $C^1$ angular divergence-free velocity of degree $1-\alpha$
equals the canonical velocity of its homogeneous curl and is included
among these data.
\end{proposition}
\begin{proof}
The number $2-\alpha$ lies strictly between zero and one, so none of
the eigenvalues $(2-\alpha)^2-k^2$ vanishes, where the
Fourier index\phantomsection\label{def:fourierindex} $k$ ranges over $\mathbb Z$.
The periodic second-order ODE therefore has a unique inverse, taking
$L^{\qw}$ to $W^{2,\qw}$. One-dimensional
Sobolev embedding makes $g$ and $g'$ bounded and continuous.
Polar differentiation gives \eqref{eq:initialvelocity} and
$\Delta\Psi_H=H$ away from zero. In integration by parts across the
circle\phantomsection\label{def:originradius}
 of radius $\rho_*>0$, the boundary terms against a fixed smooth test are estimated as
\[
\left|
\int_{\partial B_\rho}
\varphi\,\partial_r\Psi_H\,dS
\right|
=O(\rho^{2-\alpha}),\quad\text{and}\quad
\left|
\int_{\partial B_\rho}
\Psi_H\,\partial_r\varphi\,dS
\right|
=O(\rho^{3-\alpha}).
\]
They vanish as $\rho_*\downarrow0$. Therefore $\Delta\Psi_H=H$ holds on the whole plane in distributions, with no additional distribution supported at the origin.
The velocity in \eqref{eq:initialvelocity} belongs to
$\mathcal M^{\qu}$, since its angular factor is bounded and its degree
is $1-\alpha$. Its difference from $K*H$ is an entire harmonic vector
field in that same weak class. The mean-value argument used in
\eqref{eq:harmonicdecay}, together with
the Lorentz--H\"older estimate \eqref{eq:lorentzholder},
gives the difference an upper bound
\(C_{\alpha,\kappa}
 \Mnorm{\cdot}{\qu}R^{-2/\qu}
\)
for every $R>0$. Letting $R\to\infty$ makes it zero and identifies the
velocity in \eqref{eq:initialvelocity} with $K*H$.

Conversely, the curl of a $C^1$ angular homogeneous velocity has a
continuous angular vorticity $h$, hence lies in $L^{\qw}$ on the
circle. Integration by parts at the origin has vanishing boundary
terms of order $\rho_*^{2-\alpha}$, and the same harmonic uniqueness
identifies its velocity with $K*H$. A discontinuous angular $h$, or
the integrable cusp in Subsection~\ref{subsec:class}, cannot be the curl
of a velocity with a $C^1$ angular trace. The inclusion of the
stated $C^1$ datum class is therefore strict.
\end{proof}

\subsection{Basic weak existence from Choi's estimates}
The strict inclusion of the stated hypotheses
does not mean that rough-data weak Euler existence
is exclusive to our construction. In version~2 of
\cite{Choi2026}, Proposition~2.8 gives the critical profile bounds
in terms of the critical datum norm alone, and the velocity trace
argument in Lemma~4.1 is uniform under these bounds. Here is the
resulting approximation argument.
Choose\phantomsection\label{def:choiapproximation}
 smooth $h_n\to h$ in
$L^{\qw}(\mathbb S^1)$ and form the corresponding $H_n$. By
\eqref{eq:exactnorm}, $H_n\to H$ in $\mathcal M^{\qw}$, so their
canonical velocities converge in $\mathcal M^{\qu}$. Apply Choi's
theorem to these smooth angular velocities, and denote the resulting
physical velocities by $u_n$. The resulting profiles
have common critical velocity and derivative bounds. Local Rellich
compactness gives a subsequence converging strongly in $L^2_{\loc}$.
Their self-similar scalings converge strongly locally in spacetime
$L^2$ away from zero, and hence their stresses converge in local
$L^1$. The uniform weak velocity bound makes the local $L^2$ mass
of the time strip $(0,t_*)$ uniformly $\mathcal O(t_*)$, allowing the full
weak velocity identity and its initial term to pass to the limit.

More explicitly, for each compact smooth divergence-free vector test
$\boldsymbol\zeta$, the velocity equation gives
\begin{equation}\label{eq:choitrace}
 |\langle u_n(t)-K*H_n,\boldsymbol\zeta\rangle|
 \leq C_{\alpha,M_H,\boldsymbol\zeta}t.
\end{equation}
One may enlarge the constant to cover the bounded approximating
datum norms. Thus every distributional initial subsequence limit of
the limiting velocity has the same curl as $K*H$ and zero divergence.
Both belong to the global weak velocity class, so harmonic uniqueness
identifies that limit with $K*H$. Uniform local $W^{1,\ell}$ bounds,
$1<\ell<\qw$, give relative compactness in local $L^2$ and upgrade
the initial convergence to strong local $L^2$. Positive-time
continuity follows from the fixed profile and continuity of dilation
on local $L^2$. This proves basic rough-data self-similar existence from Choi's
quantitative estimates as well. The velocity compactness
argument above does not by itself prove renormalization.
If the smooth-data Euler solutions are taken to be the
renormalized limits discussed in the introduction, however,
their vorticities $\omega_n=\curl u_n$ satisfy
\[
 \omega_n(t)=H_n\circ X_{u_n}(0,t,\cdot).
\]
Let $u$ be the limiting velocity. The common critical bounds
and local spacetime convergence of $u_n$ allow us to apply
Proposition~\ref{prop:flow}. Since $H_n\to H$ in
$\mathcal M^{\qw}$, measure preservation, the local estimate
\eqref{eq:uniformintegrability}, and
Lemma~\ref{lem:composition} then give
\[
 \omega_n(t)\longrightarrow H\circ X_u(0,t,\cdot)
 \quad\hbox{in }L^p_{\loc},\qquad 1\leq p<\qw.
\]
Thus exact transport and equimeasurability also pass through
the angular approximation once this renormalization input
is included. The present construction obtains these
properties directly through the transported-profile
fixed-point problem.

\subsection{Stationary example with nonintegrable vorticity flux}
The constant angular datum $h=1$ gives a useful test of the nonlinear
formulation. By \eqref{eq:initialvelocity}, the stationary field
\begin{equation}\label{eq:radialexample}
 \omega(t,x)=|x|^{-\alpha},\qquad
 u(t,x)=\frac{|x|^{1-\alpha}}{2-\alpha}\,e_\theta
\end{equation}
has the canonical Biot--Savart velocity and is self-similar with the
same scaling. Its circular flow preserves the radial vorticity, and
the centripetal acceleration is the negative gradient of
$P(t,x)=-|x|^{2-2\alpha}/(2(\alpha-1)(2-\alpha)^2)$.
The velocity is locally $L^2$, and its circulation around radius
$\rho$ tends to zero like $\rho^{2-\alpha}$, excluding a point vortex.
For velocity, the stress flux across a circle is zero because the
velocity is tangential. The pressure boundary integral of a constant
vector test vanishes by angular integration. Subtracting that constant
leaves a term $\mathcal O(\rho^{4-2\alpha})$, which tends to zero. Thus the
identities extend across zero without a point force. Moreover
$|u\omega|=(2-\alpha)^{-1}|x|^{1-2\alpha}$, whose radial integral
near zero is proportional to $\int_0^1\rho^{2-2\alpha}\idiff\rho$.
It diverges for every $\alpha\geq3/2$, logarithmically at equality.
Thus the missing ordinary flux can actually occur within the datum
class of the theorem. The symmetrized formulation is indispensable.

\bibliographystyle{myalpha}
\bibliography{reference}
\end{document}